\documentclass[11pt,reqno]{amsart}
\usepackage[latin9,utf8]{inputenc}
\usepackage[T1]{fontenc}
\usepackage[utf8]{inputenc}
\usepackage[english]{babel}
\usepackage[margin=3cm]{geometry}
\usepackage{amsmath, amssymb}
\usepackage{amsfonts}
\usepackage{dsfont}
\usepackage{float}
\usepackage{graphicx}
\usepackage{wrapfig}
\usepackage{mathtools}
\usepackage{bbm}
\usepackage{amsthm}
\usepackage{ifthen}
\usepackage{graphicx}
\usepackage[ruled,vlined]{algorithm2e}
\usepackage{xcolor}
\usepackage{mathtools}
\usepackage{empheq}
\usepackage{mathrsfs}
\usepackage{multicol}

\newtheorem{theorem}{Theorem}[section]

\newtheorem{lemma}[theorem]{Lemma}
\newtheorem{cor}[theorem]{Corollary}

\theoremstyle{definition}
\newtheorem{definition}[theorem]{Definition}

\theoremstyle{remark}
\newtheorem{remark}{Remark}[section]

\newcommand{\bb}[1]{\mathbb{#1}}

\newcommand{\R}{\bb{R}}

\chardef\bslash=`\\ % p. 424, TeXbook
\numberwithin{equation}{section}

\newcommand{\N}{\mathbb{N}}
\newcommand{\Z}{\mathbb{Z}}

\newcommand{\Com}{\mathbb{C}}

\newcommand{\re}{\operatorname{Re}}
\newcommand{\ima}{\operatorname{Im}}

\newcommand{\sgn}{\operatorname{sgn}}

\def\bm{\left( \begin{array}{cc}}
\def\endm{\end{array}\right)}

\newcommand{\be}{\begin{equation}}
\newcommand{\ee}{\end{equation}}
\newcommand{\ba}{\left(\begin{array}{c}}
\newcommand{\ea}{\end{array}\right)}
\newcommand{\bea}{\begin{eqnarray}}
\newcommand{\eea}{\end{eqnarray}}
\newcommand{\bee}{\begin{eqnarray*}}
\newcommand{\eee}{\end{eqnarray*}}
\newcommand{\ben}{\begin{enumerate}}
\newcommand{\een}{\end{enumerate}}

\usepackage[colorlinks=true]{hyperref}
\hypersetup{linkcolor=red,citecolor=blue,filecolor=dullmagenta,urlcolor=blue}
\usepackage{chngcntr}
\numberwithin{figure}{section}

\numberwithin{equation}{section}
\usepackage{tikz}
\begin{document}
	\title[Riemann Zeta Heat flow]{The Generalized Riemann Zeta heat flow}
	\author[Castillo]{V\'ictor Castillo}  % in alphabetical order
	\address{Departamento de Matem\'aticas, Universidad de Santiago de Chile USACH.}
	\email{vacastil@mat.uc.cl}
	\thanks{V. C. was partially funded by the Chilean research grant ANID 2022 Exploration 13220060.}
	\author[Mu\~noz]{Claudio Mu\~noz}  % in alphabetical order
	\address{Departamento de Ingenier\'{\i}a Matem\'atica and Centro
de Modelamiento Matem\'atico (UMI 2807 CNRS), Universidad de Chile, Casilla
170 Correo 3, Santiago, Chile.}
	\email{cmunoz@dim.uchile.cl}
	\thanks{C.M. was partially funded by Chilean research grants ANID 2022 Exploration 13220060, FONDECYT 1191412, 1231250, and Basal CMM FB210005 and MathAmSud WAFFLE 23-MATH-18.}
	\author[Poblete]{Felipe Poblete}
\address{Instituto de Ciencias F\'isicas y Matem\'aticas, Facultad de Ciencias, Universidad Austral de Chile, Valdivia, Chile.}
\email{felipe.poblete@uach.cl}
\thanks{F.P.'s work is partially supported by ANID Exploration project 13220060, ANID project FONDECYT 1221076 and MathAmSud WAFFLE 23-MATH-18.}
\author[Salinas]{Vicente Salinas}  % in alphabetical order
	\address{Departamento de Ingenier\'{\i}a Matem\'atica, Universidad de Chile, Casilla 170 Correo 3, Santiago, Chile.}
	\email{vsalinas@dim.uchile.cl}
	\thanks{V. S. was funded by ANID 2022 Exploration 13220060, ANID Fondecyt 1231250 and Beca ANID-Subdirección de Capital Humano/Doctorado Nacional/2023-21231505.}
%\today

\maketitle

\begin{abstract}
We consider the PDE flow associated to Riemann zeta and general Dirichlet $L$-functions. These are models characterized by nonlinearities appearing in classical number theory problems, and generalizing the classical holomorphic Riemann flow studied by Broughan and Barnett. Each zero of a Dirichlet $L$-function is an exact solution of the model. In this paper, we first show local existence of bounded continuous solutions in the Duhamel sense to any Dirichlet $L$-function flow with initial condition far from the pole (as long as this exists). In a second result, we prove global existence in the case of nonlinearities of the form Dirichlet $L$-functions and data initially on the right of a possible pole at $s=1$. Additional global well-posedness and convergence results are proved in the case of the defocusing Riemann zeta nonlinearity and initial data located on the real line and close to the trivial zeros of the zeta. The asymptotic stability of any stable zero is also proved. Finally, in the Riemann zeta case, we consider the ``focusing'' model, and prove blow-up of solutions near the pole $s=1$. 
\end{abstract}
%\tableofcontents

\section{Introduction and Main Results}

\subsection{Setting}
Let $d\geq 1$ and $\N=1,2,3,\ldots$. In this paper we consider the initial value problem associated to the heat flow of the Riemann zeta function
\begin{equation}\label{eqn:PDZ}
\begin{aligned}
\partial_t u (t,x)= &~{}\Delta u (t,x)+ \lambda \zeta(u(t,x)), \quad x\in\R^d, \quad t\ge 0,\\
u(0,x) = &~{} g(x) \quad \hbox{given}.
\end{aligned}
\end{equation}
Here $u=u(t,x)\in \Com$, $\zeta$ denotes the classical Riemann zeta function, and $\lambda \in \{-1,1\}$. For some reasons to be explained below, we shall say that $\lambda=1$ corresponds to a defocusing case, and $\lambda=-1$ will represent a focusing case. The model \eqref{eqn:PDZ} is part of a larger family of PDE flows arising from number theory, represented by nonlinearities of complex-valued type referred as Dirichlet $L$-functions. % that will be considered in detail in this work. 

\medskip

The zeta function was introduced by Riemann in 1859 \cite{R59}. It originates from the classical series 
\[
\sum_{n= 1}^\infty \frac{1}{n^s},
\]
that converges if $\re(s)>1$, and diverges if $s$ approaches 1. It can be uniquely extended as a meromorphic function defined on the complex plane, with a unique pole at $s=1$. Trivial zeroes are located at $s=-2k$, $k=1,2,\ldots$.  Nontrivial zeroes of the zeta are deeply related to prime numbers. Indeed, Riemann \cite{R59} assumed in 1859 that all nontrivial zeroes are placed on the line $\re(s)=\frac{1}{2}$. This is the famous Riemann's hypothesis. This problem has attracted considerable interest from many mathematicians, although after 163 years, it remains unsolved. 

\medskip

It is well-known that the validity of the Riemann's hypothesis implies deep consequences on the distribution of prime numbers. Modifications of the zeta on varieties over finite fields and their corresponding Riemann's hypothesis have been proved true, see e.g. Deligne \cite{Del1, Del2}. In terms of numerical results, it is known that the Riemann's hypothesis is true up to a height in the imaginary variable of size $3 \times 10^{12}$ \cite{PT21}. On the other hand, Odlyzko \cite{Odl} showed that the zeroes behave very much like the eigenvalues of a random Hermitian matrix, suggesting that they are in some sense eigenvalues of an unknown self-adjoint operator. This result supports the so-called Montgomery conjecture \cite{Mon}. Finally, Bourgain obtained bounds on the growth of the $\zeta$ around the critical line \cite{Bourgain}.

\medskip

New results by Rodgers and Tao \cite{RT} showed that the De Bruijn-Newman constant $\Lambda$ is nonnegative ($\Lambda \geq 0$)\footnote{Interestingly, an ODE system related to the zeroes of the De Bruijn-Newman parametric function is related to a similar ODE system obtained when studying rational solutions to KP \cite{GPS}.}. Previously, Polymath \cite{Pol} obtained a new upper bound for the De Bruijn-Newman constant $\Lambda$ by numerical computations, improving previous foundational results by Csordas, Smith and Varga \cite{CSV}. This constant measures the validity of the Riemann's hypothesis, in the sense that $\Lambda  \leq 0$ is equivalent to the Riemann's hypothesis. So the only possibility for the Riemann's hypothesis to be true is that $\Lambda = 0$. In this sense, it is maybe false, or just barely true. Moreover, a striking property is that the particular function involved in the definition of the De Bruijn-Newman constant $\Lambda$ satisfies the \emph{backwards heat equation}. This particular finding has motivated us to study a modification of the previous model, seeking for the influence of the zeta nonlinearity in PDE models, starting with \eqref{eqn:PDZ} in more detail. 

\medskip

Another motivation to study \eqref{eqn:PDZ} comes from its ODE counterpart, that has been considered in detail some years ago. Indeed, the holomorphic Riemann flow $s'=\zeta(s)$ was worked by Broughan and Barnett in \cite{BB} (see also \cite{B0}), revealing that zeroes have different structures as critical points. Geometric properties of the holomorphic flow were shown to be equivalent to the Riemann's hypothesis. Notice that every zero of the zeta is an exact solution of each \eqref{eqn:PDZ}. Fig. \ref{Fig:1} shows the behavior of the holomorphic flow around important points such as nontrivial zeroes and the pole $s=1$.  

\medskip

\begin{figure}[htb]
   \centering
   \includegraphics[scale=0.45]{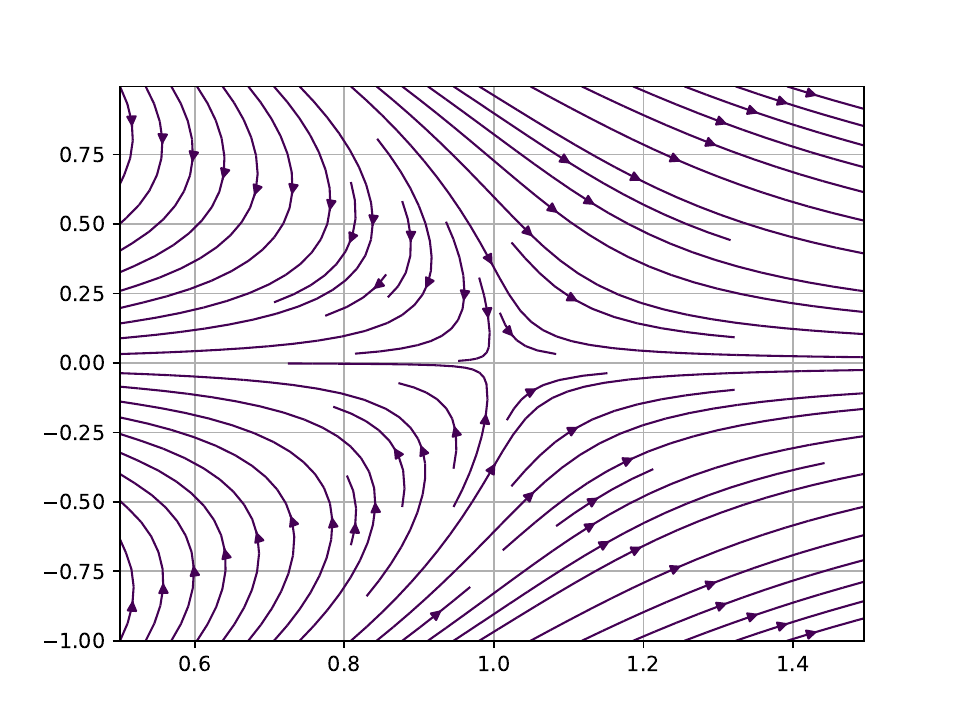}
   \includegraphics[scale=0.45]{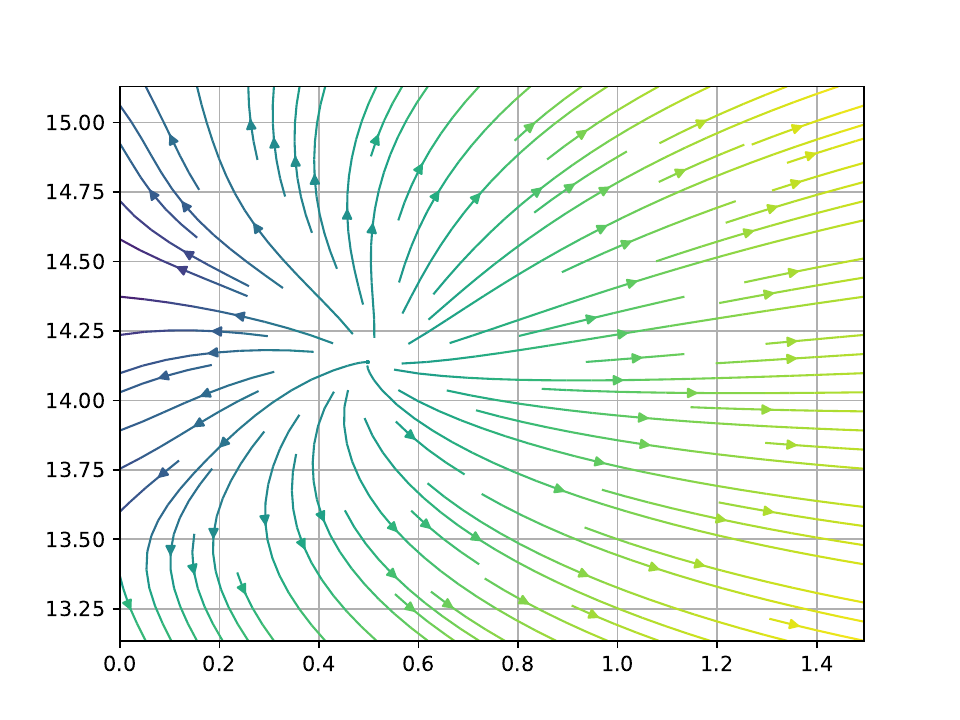} % requires the graphicx package
   \caption{Left: The holomorphic flow \cite{BB} around the pole $s=1$. Right: The same flow around the first Riemann zeta nontrivial zero $s_1\sim \frac12+14.13 i$ (unstable source).}
   \label{Fig:1}
\end{figure}

In the PDE setting, some recent advances in the study of complex valued flows show interesting features as global existence or blow up, depending on particular initial conditions. Guo, Ninomiya, Shimojo, and Yanagida \cite{GNHY} studied \eqref{eqn:PDZ} in the case where $\zeta$ is replaced by a quadratic nonlinearity, showing global existence and blow up. See also Duong \cite{Duo} for the study of \eqref{eqn:PDZ} where $\zeta$ is replaced by the meromorphic function $1/s$, and references therein for further results in this direction. 

%
%\subsection{On the literature} Polya’s approach led to the theory of totally positive functions and variation diminishing transforms. Weil’s work on the Riemann's hypothesis for curves over finite fields led to the transformation of algebraic geometry by Grothendieck. Connes showed how noncommutative geometry suggests connections between geometry, spectral theory and the Riemann's hypothesis. Similarly, we expect that the interaction of PDEs and special functions coming from analytic number theory will lead to new advances in both fields. The validity of the Riemann's hypothesis will imply important consequences in cryptography, cyber security, quantum systems, entropy of black holes, random matrices, among others.

\subsection{Dirichlet $L$-functions} Let $\chi_{m}$ be a Dirichlet character of period $m$ (see Definition \ref{def:Dirichlet} for further details), and consider for $\re(s)>1$
\begin{equation}\label{L_m_serie}
L_{m}(s)=\displaystyle{\sum_{n=1}^\infty \dfrac{\chi_{m}(n)}{n^s}},
\end{equation}
the associated Dirichlet $L$-function. As well as in the case of the Riemann zeta, $L_{m}$ may have a pole at $s=1$, it has a critical line, and it has a corresponding extension to $\re(s)<1$, making this function meromorphic on the plane in the case of a \emph{principal character}, and holomorphic in the case of a \emph{non principal character}, respectively. A particular $L$-function is the zeta itself, obtained up to a multiplicative constant with with $m=1$.

\medskip

Specifically, $L$-functions encode relevant information of objects with geometrical or arithmetical nature as elliptic curves or characters of finite groups, respectively. They are obtained as a generalization of $\zeta$. In this direction, classical extensions are the Dirichlet $L$-functions, which are related to Dirichlet characters mod $q$. They are a key piece to study the behavior of primes in arithmetic progressions. A more general point of view comes from Artin $L$-functions, which are associated to a number field $K$ and an $n$-dimensional Galois representation of the Galois group of $K/Q$, see \cite{PM1,PM2} for interesting results in this direction. The analogue of Riemann's hypothesis for $L$-functions is called Generalized Riemann Hypothesis. % (GRiemann's hypothesis).

\subsection{Main results}
The first result of this paper shows that every generalized Riemann flow has a well-posed theory for continuous and bounded initial data placed sufficiently far from the pole $s=1$.

\begin{theorem}[Local well-posedness]\label{LDirichlet}
Let $\lambda=\pm1,$ and let $L_m$ be a  Dirichlet $L$-function obtained from a principal character $m$. Consider an initial datum $g\in L^{\infty}(\R^d;\Com)\cap C(\R^d;\Com) $, $g=g_1+ig_2$, and $g_i\in\R$ be such that the uniform condition 
\begin{equation}\label{uniform_g}
\inf_{x\in \R^d} \left( |g_1(x)-1|+|g_2(x)| \right)>0,
\end{equation}
is satisfied. Then there exists a time $T>0$ such that the problem 
\begin{equation} \label{eqn:edp} 
\begin{aligned}
\partial_t u(t,x)=&~{}\Delta u(t,x)+ \lambda L_{m}(u(t,x))\quad x \in \R^d \quad t\ge 0,\\
u(0,x)=&~{}g(x),  
\end{aligned}
\end{equation}
has a unique solution $u\in C([0,T); L^{\infty}(\R^d;\Com)\cap C(\R^d;\Com) )$ under which the corresponding solution map is continuous. Moreover, if $T^*\leq +\infty$ denotes the maximal time of existence of $u$, and $T^*<+\infty$, then there exists $x_0\in\R^d$, or $x_1\in\R^d$ such that 
\begin{equation}\label{dicotomia}
\liminf_{t\uparrow T^*} |u_1(t,x_0) -1| +|u_2(t,x_0)|=0, \quad \hbox{ or } \quad \limsup_{t\uparrow T^*} |u(t,x_1)| =+\infty.
\end{equation}
\end{theorem}

The condition on $g$ \eqref{uniform_g} is natural in view that $L$-functions of principal type, the original Riemann zeta function among them, have a pole at $z=1$. Therefore, \eqref{uniform_g} prevents that for times arbitrarily close to zero one may have indefinite terms in \eqref{eqn:edp}. 

\begin{remark}
Notice that the chosen space $C((0,T]; L^{\infty}(\R^d;\Com)\cap C(\R^d;\Com) )$ is natural for  \eqref{eqn:edp}, in view that every zero of $L_{m}$ is a constant solution. Consequently, Sobolev spaces with decay at infinity are in principle only well-suited when one considers the behavior of solutions close to exact solutions, such as nontrivial zeros. 
\end{remark}

The following corollary stated without proof transpires easily from the previous result, and it is related to the case where $L_m$ is a Dirichlet $L$-function with $m$ non principal character. Notice that the condition \eqref{uniform_g} on $g$ is not needed anymore.

\begin{cor}[Non principal character case]%\label{LDirichlet_nopolo} 
Let $\lambda=\pm 1$, and  $L_{m}$ be a Dirichlet $L$-function associated to a non principal character $m$ (see Definition \ref{def:Dirichletprincipal}).  Let $g\in C(\R^d,\Com) \cap L^\infty(\R^d,\Com)$ be a given initial datum. Then the initial value problem \eqref{eqn:edp} %\eqref{eqn:edp3}:
%\[%begin{equation}\label{eqn:edp3} 
%\begin{aligned}
%\partial_t u(t,x)=&~{}\Delta u(t,x)+ \lambda L_{m}(u(t,x))\quad x \in \R^d \quad t\ge 0,\\
%u(0,x)=&~{} g(x),  
%\end{aligned}
%\]%end{equation}
has a unique solution $u\in C([0,T); L^{\infty}(\R^d;\Com)\cap C(\R^d;\Com) )$ under which the corresponding solution map is continuous. 
\end{cor}

Now we consider the problem of global existence. In order to state our result, we require a definition proposed by van de Lune in the case of general Dirichlet $L_m$ functions.

\begin{definition}[van de Lune \cite{VanLune}]\label{sigma0}
Given a general Dirichlet $L$-function $L_m$, we define $\sigma_0=\sigma_0(L_m)$ as follows:
\begin{equation}\label{sigma00}
\sigma_0 := \sup \big\{ \sigma\in\mathbb R ~ : ~ \hbox{there exists } t\in\mathbb R \hbox{ such that } \re L_m(\sigma+ it )=0 \big\}.
\end{equation}
\end{definition}

%{\color{red} 
%Creo que sería mejor considerar solo sigma en $[1,2]$, pues así queda más claro que es un conjunto acotado y no vacio, esto no cambia el valor de considerarlo sobre todo R. 

%También comentar que el 2 es un valor más cerrado, pero lo que se sabe es que es menor o igual a 1.71
%}

In other words, in the case where it is attained, $\sigma_0$ is the largest real number $\sigma$ such that the real part of $L_m$ may become zero at some point $\sigma+ it$, for some $t\in\R$. Above this point (if finite), one should have always $\re L_m \neq 0$.

\begin{remark}\label{caso_zeta_VDL}
It is well-known that in the case where $L_m= \zeta$, one has that $\sigma_0=\sigma_0(\zeta)$ as in Definition \ref{sigma0} satisfies $1<\sigma_0 <2$, and for all $t\in\R$,  $\re \zeta(\sigma + it )>0$ if $\sigma >\sigma_0$. Indeed, in this case one can take $\sigma_0\sim 1.11$ \cite{VanLune}. 
\end{remark}

\begin{remark}\label{caso_general_VDL}
Let $\sigma_1\in\R$ be such that $\zeta(\sigma_1)=2$, $\sigma_1\sim 1.71$. In the general case of a Dirichlet $L_m$, it is not difficult to show that $\sigma_0(L_m) $  in Definition \ref{sigma0} is finite and satisfies
\[
\sigma_0(L_m) \leq \sigma_1,
\]
since for all $t\in\R$, one has $\re L_m(\sigma + it ) >0$ if $\sigma>\sigma_1$. See Lemma \ref{lem:Repos} for a proof of this fact.
\end{remark}

\medskip

Our first global well-posedness result is the following. Notice that we will assume $\lambda=1$ in \eqref{eqn:edp}. Finally, for a function $g(x)=g_1(x)+ig_2(x)$, $g_1,g_2$ real-valued, consider
\begin{equation}\label{I_S}
\begin{aligned}
&~{} I_1=\inf_{x\in\R^d}g_1(x), \qquad\qquad  I_2=\inf_{x\in\R^d}g_2(x), \\
&~{}  S_1=\sup_{x\in\R^d}g_1(x), \quad \hbox{and} \quad S_2=\sup_{x\in\R^d}g_2(x).
\end{aligned}
\end{equation}

\begin{theorem}[Global well-posedness]\label{teo:globalpos}
Let $L_{m}$ be a Dirichlet $L$-function associated to a principal character $\chi_m$, and $\sigma_0(L_m)$ from \eqref{sigma00}. Let $g\in L^{\infty}(\R^d,\Com)\cap C(\R^d,\Com)$ be an initial datum of the form $g(x)=g_1(x)+ig_2(x)$, $g_1,g_2$ real-valued, and satisfying the condition
%\begin{equation}\label{cond_00}
%\sigma_0(L_m)>1, %, \quad \hbox{where} \quad \zeta(\sigma)=2. 
%\end{equation}
%and
\begin{equation}\label{cond_0}
I_1> \max\{ 1,\sigma_0(L_m)\}. %, \quad \hbox{where} \quad \zeta(\sigma)=2. 
\end{equation}	 
Let $u$ be the local solution to the problem 
\begin{equation}\label{eqn:PDZN}
\begin{aligned}
\partial_t u (t,x)= &~{}\Delta u (t,x)+ L_m(u(t,x)), \quad x\in\R^d, \quad t\ge 0,\\
u(0,x) = &~{} g(x).
\end{aligned}
\end{equation}
Then $u$ is globally well-defined. Moreover, if $u=u_1+iu_2$, $u_1,u_2$ real-valued, one has for all $t\geq 0$,
\begin{equation}\label{des_1}
\max\{0,2-\zeta(I_1)\}t+I_1\leq u_1(t,x)\leq \zeta(I_1)t+S_1,
\end{equation}
and
\begin{equation}\label{des_2}
(1-\zeta(I_1))t+I_2\leq u_2(t,x)\leq (\zeta(I_1)-1)t+S_2.
\end{equation}
\end{theorem}

Some remarks are in order:

\begin{remark}
Condition \eqref{cond_0} is equivalent to $I_1\geq \sigma_0$ in the case of $L_m=\zeta$, since $\re \zeta(1+ it )$ takes negative values for some values of $t$. Additionally, condition \eqref{cond_0} is demanded to ensure that the initial datum is on the right of the possibly existing pole and the real part of $L_m$ is positive.
\end{remark}
\begin{remark}
The term of first lower bound $2-\zeta(I_1)$ in \eqref{I_S} can be replaced by some better bound of $\re(L_m(s))$, $s=s_1+is_2,$ when $s_1\geq I_1.$
\end{remark}

\begin{remark}
Although solutions are globally defined, notice that their sizes are growing in time. This is sometimes referred as ``infinite time blow up'', in the sense that the $L^\infty$ norm of the solution continuously grows to infinity as time evolves.   
\end{remark}

An interesting consequence of the previous result is the following:
\begin{cor}[The case of real-valued characters]\label{Cor:faible}
      Let $L_m$  be a Dirichlet $L$-function obtained from a real-valued character $\chi_m$ and $\lambda=1$. Let $g(x)=g_1(x)+ig_2(x)\in L^{\infty}(\R^d,\Com)\cap C(\R^d,\Com)$ a given initial datum. If  $I_1\geq \max\{1,\sigma_0(L_m)\}$ and $I_2>0$, then the local solution to \eqref{eqn:PDZN} is global in time and satisfies 
%    \begin{equation}\label{eqn:PDZ2}
%\begin{aligned}
%\partial_t u (t,x)= &~{}\Delta u (t,x)+  L_{m}(u(t,x)), \quad x\in\R^d, \quad t\ge 0,\\
%u(0,x) = &~{} g(x).
%\end{aligned}
%\end{equation}
%We have that $u(t,x)$ the solution of equation \eqref{eqn:PDZ2} is global and satisfies 
\begin{equation}\label{eq:vieja}
\max\{0,(2-\zeta(I_1))\}t+I_1\leq u_1(t,x)\leq \zeta(I_1)t+S_1,
\end{equation}
and
\begin{equation}\label{eq:nueva}
0 < u_2(t,x)\leq (\zeta(I_1)-1)t+S_2.
\end{equation}
\end{cor}

\begin{remark}
Note that \eqref{eq:vieja} coincides with \eqref{des_1} obtained in the general case. However, \eqref{eq:nueva}, compared to \eqref{des_2}, is an improvement obtained from the additional conditions on $L_m$ (real-valued characters).
\end{remark}

It turns out that the previous global existence results can be improved if one assumes that the initial data is real-valued. Although simpler than the previous result, it is very enlightening to describe the role of the pole $s=1$ in the long time dynamics. Notice that no assumption \eqref{cond_0} is required, and $-2k$, $k=1,2,3\ldots$ are the trivial zeros of the zeta, exact solutions of \eqref{eqn:PDZN} in the case $L_m=\zeta$. Moreover, one has

\begin{remark}\label{signos}
It is well-known that in the case of $L_m=\zeta$, and $s_1\in\mathbb R$, one has for $s_1<0$
\[
\sgn\,\zeta(s_1)= \sgn \, \left(\sin\left(\frac\pi2s_1\right)\right). 
\]
In particular, $\zeta(s_1)<0$ in $(-2,0)$, $\zeta(s_1)>0$ in $(-4,-2)$, and so on. This is a consequence of the functional equation $\zeta(s)= 2^s \pi^{s-1} \Gamma(1-s) \zeta(1-s) \sin(\frac\pi2s)$.
\end{remark}

\begin{theorem}[Global well-posedness, zeta and real-valued case]\label{teo:globalReal}
    Consider $L_m=\zeta$, the classical zeta function, in equation \eqref{eqn:PDZN}. Assume now that the initial datum $g$ is real-valued and satisfies
    \begin{equation}\label{1p6_0}
    g\in L^{\infty}(\R^d)\cap C(\R^d), \quad \hbox{and} \quad \inf_{x\in \R^d}|g(x)-1|>0.
    \end{equation}
    Let
    \begin{equation}\label{I_S_new}
    S=\sup_{x\in\R^d}g(x),\quad  I=\inf_{x\in\R^d}g(x).
    \end{equation}
%    \begin{equation}\label{eqn:PDZR}
%\begin{aligned}
%\partial_t u (t,x)= &~{}\Delta u (t,x)+ \zeta(u(t,x)), \quad x\in\R^d, \quad t\ge 0,\\
%u(0,x) = &~{} g(x) \quad \hbox{given}.
%\end{aligned}
%\end{equation}
Then $u$ is global in time and the following alternative hold: 
\begin{enumerate}
\item[$(i)$] If $I>1$, then for $t\geq 0,$
\begin{equation}\label{1p6_1}
t+I\leq u(t,x)\leq \zeta(I)t+S.
\end{equation}
\item[$(ii)$] If $I<1$, we have that 
\begin{equation}\label{1p6_2}
-2k_1\leq u(t,x)\leq -2k_2,
\end{equation}
where $k_1$ and $k_2$ are such that:
%\[
%k_1:=\inf \left\{k\in \N ~ \big| ~ -2k\leq I\right\} \quad and\quad k_2:=\sup\left\{k\in \N ~ \big| ~ -2k\geq S\right\}. 
%\] 
\begin{equation}\label{k1k2}
\begin{aligned}
-2 k_1:= &~{} \max \left\{-2k \leq I ~ \big|  ~ k\in \N \right\} \\
 -2 k_2:=&~{}  \begin{cases} \inf  \left\{-2k \geq S  ~ \big| ~ k \in \N \right\} , \quad S\leq  -2,\\
 S, \quad S>-2 .
 \end{cases} 
\end{aligned}
\end{equation}
\item[$(iii)$] Additionally, if $I<1$ and if $n_1$ and $n_2$ are the unique positive integers such that $I\in (-4n_1,-4n_1+4)$, and $S\in (-4n_2,-4n_2+4)$ respectively, then
\begin{equation}\label{1p6_3}
-4n_1+2\leq \liminf_{t\to \infty }u(t,x)\leq \limsup_{t\to \infty }u(t,x)\leq -4n_2+2.
\end{equation}
\end{enumerate}
\end{theorem}
See Figures \ref{fig1} and \ref{fig2} for additional details.

\begin{figure}[htb]
\tikzset{every picture/.style={line width=0.75pt}} %set default line width to 0.75pt        
\begin{center}
\begin{minipage}{0.4\textwidth}
    
    \begin{tikzpicture}[x=0.75pt,y=0.75pt,yscale=-1,xscale=1,scale=0.4,midway]

    \draw [line width=0.75] (48,124.6) -- (640,124.6);
    \draw [shift={(643,124.6)}, rotate=180] [fill={rgb, 255:red, 0; green, 0; blue, 0 }  ][line width=0.08]  [draw opacity=0] (8.93,-4.29) -- (0,0) -- (8.93,4.29) -- cycle;
    
    \draw [draw opacity=0][line width=0.75] (129.57,154.59) .. controls (129.38,154.6) and (129.19,154.6) .. (129,154.6) .. controls (112.43,154.6) and (99,141.17) .. (99,124.6) .. controls (99,108.03) and (112.43,94.6) .. (129,94.6) -- (129,124.6) -- cycle;
    \draw [color={rgb, 255:red, 0; green, 0; blue, 0 }  ,draw opacity=1 ][line width=0.75] (129.57,154.59) .. controls (129.38,154.6) and (129.19,154.6) .. (129,154.6) .. controls (112.43,154.6) and (99,141.17) .. (99,124.6) .. controls (99,108.03) and (112.43,94.6) .. (129,94.6);
    
    \draw [draw opacity=0][line width=0.75] (329.57,154.59) .. controls (329.38,154.6) and (329.19,154.6) .. (329,154.6) .. controls (312.43,154.6) and (299,141.17) .. (299,124.6) .. controls (299,108.03) and (312.43,94.6) .. (329,94.6) .. controls (329,94.6) and (329,94.6) .. (329,94.6) -- (329,124.6) -- cycle;
    \draw [color={rgb, 255:red, 0; green, 0; blue, 0 }  ,draw opacity=1 ][line width=0.75] (329.57,154.59) .. controls (329.38,154.6) and (329.19,154.6) .. (329,154.6) .. controls (312.43,154.6) and (299,141.17) .. (299,124.6) .. controls (299,108.03) and (312.43,94.6) .. (329,94.6) .. controls (329,94.6) and (329,94.6) .. (329,94.6);
    
    \draw [draw opacity=0][line width=0.75] (268.55,94.6) .. controls (268.74,94.6) and (268.93,94.6) .. (269.12,94.6) .. controls (285.69,94.67) and (299.07,108.15) .. (299,124.72) .. controls (298.93,141.29) and (285.45,154.67) .. (268.88,154.6) -- (269,124.6) -- cycle;
    \draw [color={rgb, 255:red, 0; green, 0; blue, 0 }  ,draw opacity=1 ][line width=0.75] (268.55,94.6) .. controls (268.74,94.6) and (268.93,94.6) .. (269.12,94.6) .. controls (285.69,94.67) and (299.07,108.15) .. (299,124.72) .. controls (298.93,141.29) and (285.45,154.67) .. (268.88,154.6);
    
    \draw [draw opacity=0][line width=0.75] (516.55,94.6) .. controls (516.74,94.6) and (516.93,94.6) .. (517.12,94.6) .. controls (533.69,94.67) and (547.07,108.15) .. (547,124.72) .. controls (546.93,141.29) and (533.45,154.67) .. (516.88,154.6) .. controls (516.88,154.6) and (516.88,154.6) .. (516.88,154.6) -- (517,124.6) -- cycle;
    \draw [color={rgb, 255:red, 0; green, 0; blue, 0 }  ,draw opacity=1 ][line width=0.75] (516.55,94.6) .. controls (516.74,94.6) and (516.93,94.6) .. (517.12,94.6) .. controls (533.69,94.67) and (547.07,108.15) .. (547,124.72) .. controls (546.93,141.29) and (533.45,154.67) .. (516.88,154.6) .. controls (516.88,154.6) and (516.88,154.6) .. (516.88,154.6);
    
    \draw [color={rgb, 255:red, 208; green, 2; blue, 27 }  ,draw opacity=1 ][line width=0.75] (100,124.8) -- (192,124.8);
    \draw [shift={(196,124.8)}, rotate=180] [fill={rgb, 255:red, 208; green, 2; blue, 27 }  ,fill opacity=1 ][line width=0.08]  [draw opacity=0] (11.61,-5.58) -- (0,0) -- (11.61,5.58) -- cycle;
    
    \draw [color={rgb, 255:red, 208; green, 2; blue, 27 }  ,draw opacity=1 ][line width=0.75] (300,124.8) -- (392,124.8);
    \draw [shift={(396,124.8)}, rotate=180] [fill={rgb, 255:red, 208; green, 2; blue, 27 }  ,fill opacity=1 ][line width=0.08]  [draw opacity=0] (11.61,-5.58) -- (0,0) -- (11.61,5.58) -- cycle;
    
    \draw [color={rgb, 255:red, 208; green, 2; blue, 27 }  ,draw opacity=1 ][line width=0.75] (297,124.8) -- (206,124.8);
    \draw [shift={(202,124.8)}, rotate=360] [fill={rgb, 255:red, 208; green, 2; blue, 27 }  ,fill opacity=1 ][line width=0.08]  [draw opacity=0] (11.61,-5.58) -- (0,0) -- (11.61,5.58) -- cycle;
    
    \draw [color={rgb, 255:red, 208; green, 2; blue, 27 }  ,draw opacity=1 ][line width=0.75] (546,124.8) -- (407,124.8);
    \draw [shift={(403,124.8)}, rotate=360] [fill={rgb, 255:red, 208; green, 2; blue, 27 }  ,fill opacity=1 ][line width=0.08]  [draw opacity=0] (11.61,-5.58) -- (0,0) -- (11.61,5.58) -- cycle;
    
    \draw [draw opacity=0][line width=0.75] (68.55,94.6) .. controls (68.74,94.6) and (68.93,94.6) .. (69.12,94.6) .. controls (85.69,94.67) and (99.07,108.15) .. (99,124.72) .. controls (98.93,141.29) and (85.45,154.67) .. (68.88,154.6) -- (69,124.6) -- cycle;
    \draw [color={rgb, 255:red, 0; green, 0; blue, 0 }  ,draw opacity=1 ][line width=0.75] (68.55,94.6) .. controls (68.74,94.6) and (68.93,94.6) .. (69.12,94.6) .. controls (85.69,94.67) and (99.07,108.15) .. (99,124.72) .. controls (98.93,141.29) and (85.45,154.67) .. (68.88,154.6);
    
    \draw [color={rgb, 255:red, 208; green, 2; blue, 27 }  ,draw opacity=1 ][line width=0.75] (97,124.8) -- (48,124.8);
    \draw [shift={(44,124.8)}, rotate=360] [fill={rgb, 255:red, 208; green, 2; blue, 27 }  ,fill opacity=1 ][line width=0.08]  [draw opacity=0] (11.61,-5.58) -- (0,0) -- (11.61,5.58) -- cycle;
    
    \draw [fill={rgb, 255:red, 208; green, 2; blue, 27 }  ,fill opacity=1 ] (195,124.7) circle (2);
    \draw [fill={rgb, 255:red, 208; green, 2; blue, 27 }  ,fill opacity=1 ] (396,124.7) circle (2);
    \draw [fill={rgb, 255:red, 0; green, 0; blue, 0 }  ,fill opacity=1 ] (95,124.7) circle (2);
    \draw [fill={rgb, 255:red, 0; green, 0; blue, 0 }  ,fill opacity=1 ] (295,124.7) circle (2);
    
    \draw (370,150) node [anchor=north west][inner sep=0.75pt] {\tiny $-2$};
    \draw (270,150) node [anchor=north west][inner sep=0.75pt] {\tiny $-4$};
    \draw (170,150) node [anchor=north west][inner sep=0.75pt] {\tiny $-6$};
    \draw (70,150) node [anchor=north west][inner sep=0.75pt] {\tiny $-8$};
    \draw (535,150) node [anchor=north west][inner sep=0.75pt] {\tiny $1$};
    \draw (480,150) node [anchor=north west][inner sep=0.75pt] {\tiny $0$};
    \draw (615,130) node [anchor=north west][inner sep=0.75pt] {$\mathbb{R}$};
\end{tikzpicture}
\end{minipage}\hspace{1cm}
\begin{minipage}{0.4\textwidth}
 
\begin{tikzpicture}[x=0.75pt,y=0.75pt,yscale=-1,xscale=1,scale=0.4]
%uncomment if require: \path (0,453); %set diagram left start at 0, and has height of 453

%Curve Lines [id:da4767431026201303] 
\draw [color={rgb, 255:red, 155; green, 155; blue, 155 }  ,draw opacity=1 ][line width=0.75]    (3,277) .. controls (15.38,267.71) and (11.58,204.88) .. (90,207) .. controls (168.42,209.12) and (140.06,383.25) .. (205,394.8) .. controls (269.94,406.35) and (256,217) .. (307,228) .. controls (358,239) and (342.02,328.21) .. (378,335) .. controls (413.98,341.79) and (447.02,244.03) .. (477,236) .. controls (506.98,227.97) and (515.84,371.71) .. (542,370) .. controls (568.16,368.29) and (617.17,276.82) .. (622,273.2) ;
%Straight Lines [id:da5070042442908465] 
\draw [color={rgb, 255:red, 208; green, 2; blue, 27 }  ,draw opacity=1 ][line width=0.75]    (0,200) -- (651,200.2) ;
%Straight Lines [id:da6732847302508222] 
\draw [color={rgb, 255:red, 208; green, 2; blue, 27 }  ,draw opacity=1 ][line width=0.75]    (0,400) -- (651,400.2) ;
%Straight Lines [id:da30107643680422846] 
\draw    (300.99,101) -- (300,433.8) (304.84,151.01) -- (296.84,150.99)(304.69,201.01) -- (296.69,200.99)(304.54,251.01) -- (296.54,250.99)(304.4,301.01) -- (296.4,300.99)(304.25,351.01) -- (296.25,350.99)(304.1,401.01) -- (296.1,400.99) ;
\draw [shift={(301,98)}, rotate = 90.17] [fill={rgb, 255:red, 0; green, 0; blue, 0 }  ][line width=0.08]  [draw opacity=0] (8.93,-4.29) -- (0,0) -- (8.93,4.29) -- cycle    ;
%Straight Lines [id:da40432880636038426] 
\draw    (-2,150.2) -- (650,150.2) ;
\draw [shift={(652,150.2)}, rotate = 180] [fill={rgb, 255:red, 0; green, 0; blue, 0 }  ][line width=0.08]  [draw opacity=0]   (10.93,-3.29) --   (0,0) -- (10.93,3.29) -- cycle  ;

% Text Node
\draw (627,110) node [anchor=north west][inner sep=0.75pt]    {$\mathbb{R}$};
% Text Node
\draw (310,190.2) node [anchor=north west][inner sep=0.75pt]    {\tiny ${-2}$};
% Text Node
\draw (310,241.2) node [anchor=north west][inner sep=0.75pt]    {\tiny ${-4}$};
% Text Node
\draw (310,290.2) node [anchor=north west][inner sep=0.75pt]    {\tiny ${-6}$};
% Text Node
\draw (310,340.2) node [anchor=north west][inner sep=0.75pt]    {\tiny ${-8}$};
% Text Node
\draw (310,390.2) node [anchor=north west][inner sep=0.75pt]    {\tiny ${-10}$};
% Text Node
\draw (60,88.4) node [anchor=north west][inner sep=0.75pt]    {$\textcolor[rgb]{0.61,0.61,0.61}{\lim\limits_{t\ \rightarrow \infty }}\textcolor[rgb]{0.61,0.61,0.61}{u}\textcolor[rgb]{0.61,0.61,0.61}{(}\textcolor[rgb]{0.61,0.61,0.61}{t,x}\textcolor[rgb]{0.61,0.61,0.61}{)}$};

\end{tikzpicture}
\end{minipage}
\end{center}
\caption{Left: Sets considered in Theorem \ref{Domaint} $(ii)$. Red points are sinks, while black ones sources. Right: A scheme of the long time behavior of $u$ in  Theorem \ref{Domaint} $(ii)$ in the case $g(x)\in (-10,-2)$.}\label{fig1}
\end{figure}
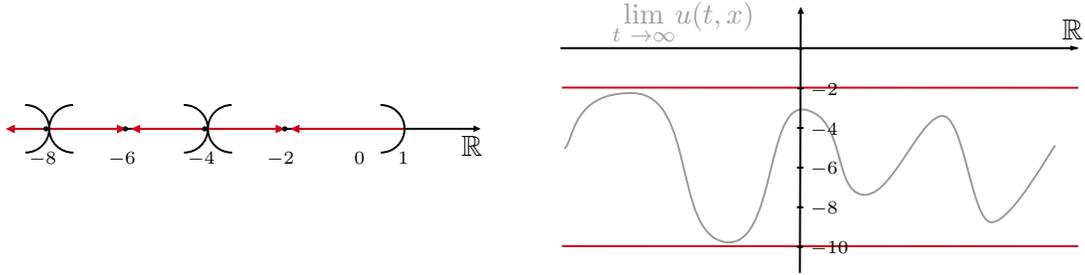

\begin{remark}
From Section 2 of \cite{BB} one  can see that $-4n_i+2$, $i\in\{1,2\}$ represents a trivial stable zero while $-4n_i$, $i\in\{1,2\}$  a  trivial unstable zero. In particular, if $n_1=n_2$, one gets for each $x\in \mathbb{R}^d$,
\begin{equation*}
  \lim_{t\to \infty }u(t,x)=-4n_1+2.
\end{equation*}
\end{remark}

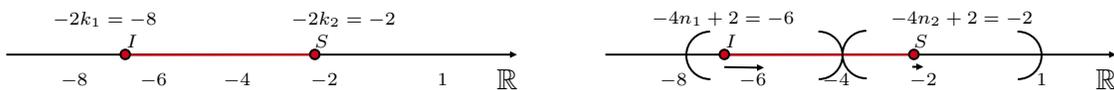
\begin{figure}

        \tikzset{every picture/.style={line width=0.75pt}} %set default line width to 0.75pt  
      \begin{tikzpicture}[x=0.75pt,y=0.75pt,yscale=-1,xscale=1,scale=0.4]
    %uncomment if require: \path (0,453); %set diagram left start at 0, and has height of 453
  
%uncomment if require: \path (0,454); %set diagram left start at 0, and has height of 454

%Straight Lines [id:da6461721518215542] 
\draw [line width=0.75]    (0,249.6) -- (641,249.6) ;
\draw [shift={(644,249.6)}, rotate = 180] [fill={rgb, 255:red, 0; green, 0; blue, 0 }  ][line width=0.08]  [draw opacity=0] (8.93,-4.29) -- (0,0) -- (8.93,4.29) -- cycle    ;
%Shape: Boxed Line [id:dp516957224887697] 
\draw [color={rgb, 255:red, 208; green, 2; blue, 27 }  ,draw opacity=1 ][line width=0.75]    (152,250.12) -- (389,249.8) ;
%Shape: Circle [id:dp576039295001854] 
\draw  [fill={rgb, 255:red, 208; green, 2; blue, 27 }  ,fill opacity=1 ] (144,249.4) .. controls (144,246.09) and (146.69,243.4) .. (150,243.4) .. controls (153.31,243.4) and (156,246.09) .. (156,249.4) .. controls (156,252.71) and (153.31,255.4) .. (150,255.4) .. controls (146.69,255.4) and (144,252.71) .. (144,249.4) -- cycle ;
%Shape: Circle [id:dp5781586239820675] 
\draw  [fill={rgb, 255:red, 208; green, 2; blue, 27 }  ,fill opacity=1 ] (383,249.4) .. controls (383,246.09) and (385.69,243.4) .. (389,243.4) .. controls (392.31,243.4) and (395,246.09) .. (395,249.4) .. controls (395,252.71) and (392.31,255.4) .. (389,255.4) .. controls (385.69,255.4) and (383,252.71) .. (383,249.4) -- cycle ;

% Text Node
\draw (380,270) node [anchor=north west][inner sep=0.75pt]    {\tiny $-2$};
% Text Node
\draw (270,270) node [anchor=north west][inner sep=0.75pt]    {\tiny $-4$};
% Text Node
\draw (165,270) node [anchor=north west][inner sep=0.75pt]    {\tiny $-6$};
% Text Node
\draw (65,270) node [anchor=north west][inner sep=0.75pt]    {\tiny $-8$};
% Text Node
\draw (540,270) node [anchor=north west][inner sep=0.75pt]    {\tiny $1$};
% Text Node
\draw (148,220.4) node [anchor=north west][inner sep=0.75pt]    {\tiny $I$};
% Text Node
\draw (385,221.4) node [anchor=north west][inner sep=0.75pt]    {\tiny $S$};
% Text Node
\draw (55,191.4) node [anchor=north west][inner sep=0.75pt]    {\tiny $-2k_{1} =-8$};
% Text Node
\draw (356,191.4) node [anchor=north west][inner sep=0.75pt]    {\tiny $-2k_{2} =-2$};
% Text Node
\draw (614.6,264.4) node [anchor=north west][inner sep=0.75pt]    {$\mathbb{R}$};
\end{tikzpicture}\hspace{1cm}
\begin{tikzpicture}[x=0.75pt,y=0.75pt,yscale=-1,xscale=1,scale=0.4]
%Straight Lines [id:da9841492308862272] 
\draw [line width=0.75]    (0,249.6) -- (641,249.6) ;
\draw [shift={(644,249.6)}, rotate = 180] [fill={rgb, 255:red, 0; green, 0; blue, 0 }  ][line width=0.08]  [draw opacity=0] (8.93,-4.29) -- (0,0) -- (8.93,4.29) -- cycle    ;
%Shape: Boxed Line [id:dp5657183900498219] 
\draw [color={rgb, 255:red, 208; green, 2; blue, 27 }  ,draw opacity=1 ][line width=0.75]    (152,250.12) -- (389,249.8) ;
%Shape: Circle [id:dp21305613110816468] 
\draw  [fill={rgb, 255:red, 208; green, 2; blue, 27 }  ,fill opacity=1 ] (144,249.4) .. controls (144,246.09) and (146.69,243.4) .. (150,243.4) .. controls (153.31,243.4) and (156,246.09) .. (156,249.4) .. controls (156,252.71) and (153.31,255.4) .. (150,255.4) .. controls (146.69,255.4) and (144,252.71) .. (144,249.4) -- cycle ;
%Shape: Circle [id:dp9921184012839073] 
\draw  [fill={rgb, 255:red, 208; green, 2; blue, 27 }  ,fill opacity=1 ] (383,249.4) .. controls (383,246.09) and (385.69,243.4) .. (389,243.4) .. controls (392.31,243.4) and (395,246.09) .. (395,249.4) .. controls (395,252.71) and (392.31,255.4) .. (389,255.4) .. controls (385.69,255.4) and (383,252.71) .. (383,249.4) -- cycle ;
%Shape: Arc [id:dp11443486753542986] 
\draw  [draw opacity=0][line width=0.75]  (132.57,280.59) .. controls (132.38,280.6) and (132.19,280.6) .. (132,280.6) .. controls (115.43,280.6) and (102,267.17) .. (102,250.6) .. controls (102,234.03) and (115.43,220.6) .. (132,220.6) -- (132,250.6) -- cycle ; \draw  [color={rgb, 255:red, 0; green, 0; blue, 0 }  ,draw opacity=1 ][line width=0.75]  (132.57,280.59) .. controls (132.38,280.6) and (132.19,280.6) .. (132,280.6) .. controls (115.43,280.6) and (102,267.17) .. (102,250.6) .. controls (102,234.03) and (115.43,220.6) .. (132,220.6) ;  
%Shape: Arc [id:dp14590213491353143] 
\draw  [draw opacity=0][line width=0.75]  (329.57,279.59) .. controls (329.38,279.6) and (329.19,279.6) .. (329,279.6) .. controls (312.43,279.6) and (299,266.17) .. (299,249.6) .. controls (299,233.03) and (312.43,219.6) .. (329,219.6) .. controls (329,219.6) and (329,219.6) .. (329,219.6) -- (329,249.6) -- cycle ; \draw  [color={rgb, 255:red, 0; green, 0; blue, 0 }  ,draw opacity=1 ][line width=0.75]  (329.57,279.59) .. controls (329.38,279.6) and (329.19,279.6) .. (329,279.6) .. controls (312.43,279.6) and (299,266.17) .. (299,249.6) .. controls (299,233.03) and (312.43,219.6) .. (329,219.6) .. controls (329,219.6) and (329,219.6) .. (329,219.6) ;  
%Shape: Arc [id:dp6931153464110213] 
\draw  [draw opacity=0][line width=0.75]  (268.55,220.6) .. controls (268.74,220.6) and (268.93,220.6) .. (269.12,220.6) .. controls (285.69,220.67) and (299.07,234.15) .. (299,250.72) .. controls (298.93,267.29) and (285.45,280.67) .. (268.88,280.6) -- (269,250.6) -- cycle ; \draw  [color={rgb, 255:red, 0; green, 0; blue, 0 }  ,draw opacity=1 ][line width=0.75]  (268.55,220.6) .. controls (268.74,220.6) and (268.93,220.6) .. (269.12,220.6) .. controls (285.69,220.67) and (299.07,234.15) .. (299,250.72) .. controls (298.93,267.29) and (285.45,280.67) .. (268.88,280.6) ;  
%Shape: Arc [id:dp5904280645853985] 
\draw  [draw opacity=0][line width=0.75]  (519.55,220.6) .. controls (519.74,220.6) and (519.93,220.6) .. (520.12,220.6) .. controls (536.69,220.67) and (550.07,234.15) .. (550,250.72) .. controls (549.93,267.29) and (536.45,280.67) .. (519.88,280.6) .. controls (519.88,280.6) and (519.88,280.6) .. (519.88,280.6) -- (520,250.6) -- cycle ; \draw  [color={rgb, 255:red, 0; green, 0; blue, 0 }  ,draw opacity=1 ][line width=0.75]  (519.55,220.6) .. controls (519.74,220.6) and (519.93,220.6) .. (520.12,220.6) .. controls (536.69,220.67) and (550.07,234.15) .. (550,250.72) .. controls (549.93,267.29) and (536.45,280.67) .. (519.88,280.6) .. controls (519.88,280.6) and (519.88,280.6) .. (519.88,280.6) ;  
%Straight Lines [id:da3427417387588372] 
\draw    (150,265.4) -- (193.02,266.05) -- (197,266.02) ;
\draw [shift={(200,266)}, rotate = 179.55] [fill={rgb, 255:red, 0; green, 0; blue, 0 }  ][line width=0.08]  [draw opacity=0] (8.93,-4.29) -- (0,0) -- (8.93,4.29) -- cycle    ;
%Straight Lines [id:da8025384593136264] 
\draw    (387.02,265.05) -- (398,265.01) ;
\draw [shift={(401,265)}, rotate = 179.78] [fill={rgb, 255:red, 0; green, 0; blue, 0 }  ][line width=0.08]  [draw opacity=0] (8.93,-4.29) -- (0,0) -- (8.93,4.29) -- cycle    ;

% Text Node
\draw (380,270) node [anchor=north west][inner sep=0.75pt]    {\tiny $-2$};
% Text Node
\draw (270,270) node [anchor=north west][inner sep=0.75pt]    {\tiny $-4$};
% Text Node
\draw (165,270) node [anchor=north west][inner sep=0.75pt]    {\tiny $-6$};
% Text Node
\draw (65,270) node [anchor=north west][inner sep=0.75pt]    {\tiny $-8$};
% Text Node
\draw (540,270) node [anchor=north west][inner sep=0.75pt]    {\tiny $1$};
% Text Node
\draw (148,220.4) node [anchor=north west][inner sep=0.75pt]    {\tiny $I$};
% Text Node
\draw (385,221.4) node [anchor=north west][inner sep=0.75pt]    {\tiny $S$};
% Text Node
\draw (55,191.4) node [anchor=north west][inner sep=0.75pt]    {\tiny $-4n_{1} +2=-6$};
% Text Node
\draw (356,191.4) node [anchor=north west][inner sep=0.75pt]    {\tiny $-4n_{2} +2=-2$};
% Text Node
\draw (614.6,264.4) node [anchor=north west][inner sep=0.75pt]    {$\mathbb{R}$};

\end{tikzpicture}

\caption{Left: Example of a real-valued initial condition $g$ is located on the region determined by $I$ and $S$, in this case, above -8 and below -2. Right: Evolution of the solution $u$ in the previous setting. It is shown in Theorem \ref{teo:globalReal} $(iii)$ that $\lim_{t\to+\infty} u$ will stay between $-6$ by below and $-2$ by above.}\label{fig2}
\end{figure}

%Now we discuss the stability of general Riemann zeros, seen as global exact solutions of the Riemann zeta flow.  Recall that from \cite{BB} stable zeros $z_0$ of the holomorphic Riemann zeta flow $s'(t)=\zeta(s(t))$ are characterized by the sign condition $\re \zeta'(z_0)<0.$
%
%\begin{theorem}[Asymptotic stability of stable zeros]\label{ThmZS}  Consider $L_m=\zeta$, the classical zeta function, in equation \eqref{eqn:PDZN}. Let $z_0$ be a stable zero of $\zeta$. There exists $\delta=\delta(\re\zeta'(z_0)) > 0$ such that for any bounded continuous initial datum $g$ in the disc $D(z_0, \delta)$, the associated solution $u$ is global in time and for each $x \in \mathbb{R}^d$,
%\begin{equation*}
%  \lim_{t\to \infty }u(t,x)=z_0.
%\end{equation*}
%\end{theorem}

Now we discuss the stability of general Riemann zeros, seen as global exact solutions of the Riemann zeta flow.  Recall that from \cite{BB} stable zeros $z_0$ of the holomorphic Riemann zeta flow $s'(t)=\zeta(s(t))$ are characterized by the sign condition $\re \zeta'(z_0)<0.$ In particular, this zero is nondegenerate.

\begin{theorem}[Asymptotic stability of stable zeros]\label{ThmZS}  Consider $L_m=\zeta$, the classical zeta function, in equation \eqref{eqn:PDZN}. Let $z_0$ be a stable zero of $\zeta$. There exists $\delta=\delta(\re\zeta'(z_0)) > 0$ such that for any bounded continuous initial datum $g$ in the disc $D(z_0, \delta)$, the associated solution $u$ is global in time and for each $x \in \mathbb{R}^d$,
\begin{equation*}
  \lim_{t\to \infty }u(t,x)=z_0.
\end{equation*}
\end{theorem}

The previous result shows that every nontrivial zero of the Riemann zeta flow is asymptotically stable, showing that the behavior of the holomorphic flow can be extended to the space-time PDE setting. The case of unstable zeros requires more work since separatrices are present, see \cite{BB} for more details. Additionally, Theorem \ref{ThmZS} can be guessed by the natural identity (valid for sufficiently decaying data $u$ and $\partial_t u$)
\[
\frac{d}{dt}\frac12\int_{\mathbb R^d} |\partial_t u|^2 =- \int_{\mathbb R^d} |\nabla \partial_t u|^2 +\int_{\mathbb R^d}  \re L_m'(u)|\partial_t u|^2,
\]
obtained by formally taking time derivative in \eqref{eqn:PDZN}, testing against $\overline{\partial_t u},$ and recalling that $\re L_m'(u)=\re \zeta'(u)$ should not change sign in the vicinity of a zeta zero.

\medskip

Some important remarks concerning the zeros of the $\zeta$ as critical points of the Riemann zeta flow are necessary.

\begin{remark}[On stable and unstable critical points of the holomorphic flow]
From the holomorphic flow case \cite[Theorem 4.2]{BB}, the stability of trivial zeros is well-established. Indeed, (unstable) sources alternate with (stable) sinks, thanks to the identities
\[
    \zeta^{\prime}(-2 n)=(-1)^n \frac{n(2 n-1) !}{(2 \pi)^{2 n}} \zeta(2 n+1), \quad n \in \{1,2,\ldots\}.
\]
Nontrivial simple zeros always behave locally as focii \cite{BB}.  Clearly, their stability is complicated. In \cite{BB}, it was conjectured that there exist an infinite number of sources and an infinite number of sinks on the critical line. Apparently, sinks represent a small proportion compared to sources, as seen in the encoding of the first 500 critical zeros provided in \cite[Table 2]{BB}. We conjecture that the proportion of sinks vs. nontrivial zeros follows a logarithmic distribution of the form $P_n\sim a\log(n)+b $, for some $a,b\in \R$ and $n \leq n_0$ fixed sufficiently large. See Fig. \ref{Fig:cuatica} for additional details. In the case of nontrivial zeros outside the critical line but inside the critical strip, of the form $z=\frac12 +x+iy$, $x,y>0$, there is a symmetric zero $z=\frac12 - x+iy$, and both share the same ODE analysis \cite{BB}.
\end{remark}

\begin{center}
\begin{figure}[h]   
    \includegraphics[scale=0.4]{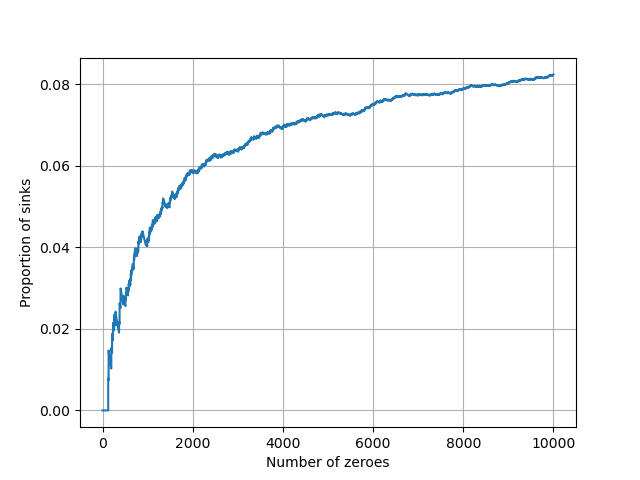}
    \includegraphics[scale=0.4]{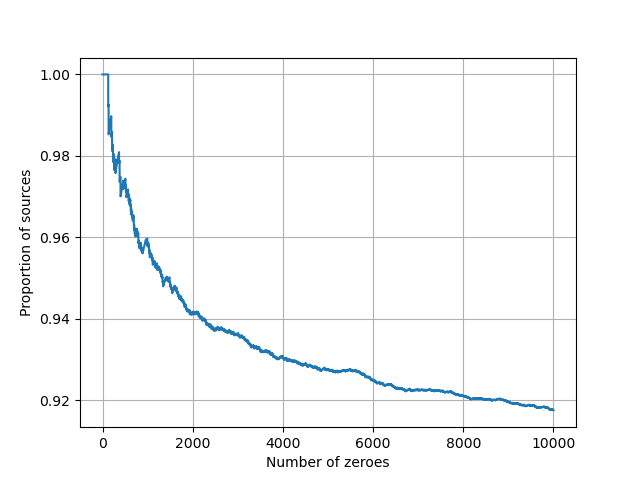}    
    \caption{\small Left: Distribution graph depicting the proportion sinks vs. nontrivial zeros among the first 10000 nontrivial zeros on the critical line. The $y$-axis describes the proportion 
    \[
    P_n = \dfrac{\# \text{sinks of the first } n \text{ nontrivial zeros}}{n},
    \]
    where $n$ represents the first $n$ nontrivial zeros (on the $x$-axis). Right: The proportion sources vs. nontrivial zeros.}\label{Fig:cuatica}
\end{figure}
\end{center}

Finally, our last result concerns the blow up problem. Notice that in this particular case we have chosen $\lambda=-1$, emulating the defocusing/focusing dichotomy presented in many nonlinear models. In this case, the defocusing/focusing character is instead related to an attractor character played by the (possible) pole at $s=1$ (see Fig. \ref{fig3}).

\begin{theorem}[Blow-up, real-value case]\label{ThmC5}
Let $L_m=\zeta$ and $\lambda=-1$. Let $g(x)\in L^{\infty}(\R^d)\cap C(\R^d)$ denote a real-valued initial datum, and consider $I$ and $S$ as in \eqref{I_S_new}.  If $I>1$, or if $-2 <I <S <1$, then the solution to
\begin{equation}\label{eqn:PDZR-}
\begin{aligned}
\partial_t u (t,x)= &~{}\Delta u (t,x)- \zeta(u(t,x)), \quad x\in\R^d, \quad t\ge 0,\\
u(0,x) = &~{} g(x),
\end{aligned}
\end{equation}
must blow-up in finite time. Finally, if $S<-2$, $u$ is global in time.
\end{theorem}
\begin{figure}

\begin{center}

    \tikzset{every picture/.style={line width=0.75pt}} %set default line width to 0.75pt        

    \begin{tikzpicture}[x=0.75pt,y=0.75pt,yscale=-1,xscale=1]
    %uncomment if require: \path (0,453); %set diagram left start at 0, and has height of 453
    
    %Straight Lines [id:da6631541560086549] 
    \draw    (99,250.4) -- (301,250.62) -- (462,250.8) ;
    \draw [shift={(465,250.8)}, rotate = 180.06] [fill={rgb, 255:red, 0; green, 0; blue, 0 }  ][line width=0.08]  [draw opacity=0] (8.93,-4.29) -- (0,0) -- (8.93,4.29) -- cycle    ;
    %Straight Lines [id:da9729496990332198] 
    \draw [color={rgb, 255:red, 155; green, 155; blue, 155 }  ,draw opacity=1 ][line width=0.75]    (150,251) -- (292,250.81) ;
    \draw [shift={(296,250.8)}, rotate = 179.92] [fill={rgb, 255:red, 155; green, 155; blue, 155 }  ,fill opacity=1 ][line width=0.08]  [draw opacity=0] (11.61,-5.58) -- (0,0) -- (11.61,5.58) -- cycle    ;
    %Shape: Circle [id:dp4023376424852372] 
    \draw  [fill={rgb, 255:red, 208; green, 2; blue, 27 }  ,fill opacity=1 ] (297,250.7) .. controls (297,248.66) and (298.66,247) .. (300.7,247) .. controls (302.74,247) and (304.4,248.66) .. (304.4,250.7) .. controls (304.4,252.74) and (302.74,254.4) .. (300.7,254.4) .. controls (298.66,254.4) and (297,252.74) .. (297,250.7) -- cycle ;
    %Straight Lines [id:da28369302310172273] 
    \draw [color={rgb, 255:red, 155; green, 155; blue, 155 }  ,draw opacity=1 ][line width=0.75]    (450,251.2) -- (309,250.81) ;
    \draw [shift={(305,250.8)}, rotate = 0.16] [fill={rgb, 255:red, 155; green, 155; blue, 155 }  ,fill opacity=1 ][line width=0.08]  [draw opacity=0] (11.61,-5.58) -- (0,0) -- (11.61,5.58) -- cycle    ;
    
    % Text Node
    \draw (296,259.4) node [anchor=north west][inner sep=0.75pt]    {$1$};
    % Text Node
    \draw (137,259.4) node [anchor=north west][inner sep=0.75pt]    {$-2$};
    % Text Node
    \draw (245,258.4) node [anchor=north west][inner sep=0.75pt]    {$0$};
    % Text Node
    \draw (445,259.2) node [anchor=north west][inner sep=0.75pt]    {$\mathbb{R}$};

    \end{tikzpicture}
\end{center}
\caption{Blow up (a.k.a. quenching) gathered by the pole $s=1$ as the attractor of real-valued trajectories placed in its vicinity.}\label{fig3}
\end{figure}
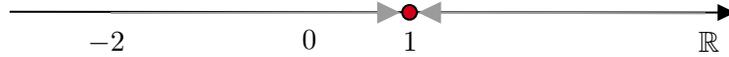
Here, we must precise the notion of blow up that we invoke. Naturally, given the space where solutions are defined (continuous and bounded), equation \eqref{eqn:PDZR-} makes no sense if $u$ touches the value 1 at some space-time point $(t,x)$. In that sense, we prove that there exists a time $t_0>0$ and a point $x_0\in \R^d$ under which $\limsup_{t\uparrow t_0} u(t,x_0)=1$.

\subsection*{Organization of this paper} This work is organized as follows: In Section \ref{sect:2} we introduce and recall the main tools Number Theory needed for the conception of main results.  Section \ref{sec:2b} states and proves several comparison principles needed along this work. Section \ref{sect:3} is devoted to the proof of uniform bounds in the Riemann zeta, Hurwitz zeta, and general Dirichlet $L$-functions. Section \ref{sect:4} contains the proof of local existence and well-posedness, Theorem \ref{LDirichlet}. In Section \ref{sec:5} we prove the global existence results, Theorems \ref{teo:globalpos} and \ref{teo:globalReal}, and Corollary \ref{Cor:faible}. Section \ref{estabilidad} is devoted to the proof of the asymptotic stability of stable zeros, Theorem \ref{ThmZS}. Finally, Section \ref{sec:6} contains the proof of blow-up, Theorem \ref{ThmC5}.

%{\color{blue}
%Ahora con los cambios podriamos dejar el \eqref{ThmZS} en la Section \ref{sec:6} 
%}
\subsection*{Acknowledgements} This work was supported by a national contract ``Exploraci\'on ANID No. 13220060''.  The authors thank Professors Eduardo Friedman, Juan Carlos Pozo and Felipe Gonçalves for comments and suggestions. Prof. Friedman suggestions and key comments concerning the behavior of zeta functions are deeply appreciated. Prof. Gonçalves discussions and comments while visiting DIM at U. Chile are also very much appreciated.  Part of this work was done while C.M. and V. S. were visiting Universidad Austral in Valdivia Chile, as invited speakers of the ``First Symposium on PDEs and Number Theory''. Part of this work was done while V. S. was visiting profs. Felipe Linares and Felipe Gonçalves at IMPA, Rio de Janeiro Brazil. He deeply thanks IMPA and both professors for enlightening discussions and support in the conception of this work and other projects.

\section{Preliminaries}\label{sect:2}

\subsection{Banach spaces}%\label{sect:21} 
We start with the definitions that we will use in this work.

\begin{definition}\label{def:YX} 
Let $Y:=L^{\infty}(\R^d;\Com)\cap C(\R^d;\Com)$ the space of bounded continuous functions  from $\R^d$  to the complex numbers. For each $g\in Y$ we write $g=g_1+ig_2$, where $g_1$ and $g_2$ are real-valued. For a $g\in Y$ we introduce the norm:  
\[
\|g\|_Y=\max\{\|g_1\|_{L^{\infty}},\|g_2\|_{L^{\infty}}\}.
\]
It is direct to check that $(Y,\|\cdot\|_Y)$ is a Banach space. % {is complete}), because it is a closed subspaces of a Banach.

\medskip

For   $T>0$ fixed, we denote  $X:=X_T=C([0,T],Y)$  the space of continuous functions from $[0,T]$ with values in $Y$, endowed with the supremum  norm:
\[
\|u\|_X=\sup_{t\in [0,T]}\|u(t)\|_Y. 
\]
Since $[0,T]$ is compact and $Y$ complete, one easily have that $(X,\|\cdot\|_X)$ defines a Banach vector space. 
\end{definition}

\begin{definition}%\label{def:P} 
Let $P:Y\to Y$ be the nonlinear mapping defined as for $h=h_1+ih_2$,
\begin{equation}\label{def:P}
P(h)=|h_1-1|+|h_2|.
\end{equation}
\end{definition}

It is important to stress that $P(h)$ is a measure of how far from $s=1$ is $h$, in the case where $L_m$ has a pole at $s=1$. In the case where there is no pole, $P$ will not be necessary. The main obstacle to get a defined zeta flow will naturally be the possibility of touching the pole at a certain time.

\subsection{General Hurwitz zeta functions} It turns out that in order to describe the long time behavior in the case of general $L$-functions, we will need information about general zeta functions, in particular, a class including the standard Riemann zeta function.

\medskip

Let $\alpha\not\in \{\ldots, -3,-2,-2,0\}$. The Hurwitz zeta function $\zeta(s,\alpha)$  is defined (see \cite{Ada07}) by
\begin{equation}\label{zeta_general}
\zeta(s,\alpha) := \sum_{n\geq 0} \frac1{(\alpha+n)^s}, \quad \re(s) >1. %, \quad \alpha\not\in \{\ldots, -3,-2,-2,0\}.
\end{equation}
 This series converges absolutely for $\re(s) >1$ and uniform  in the complex subset described by $\re(s) \geq 1+\delta$ where $\delta>0$. Also the Hurwitz zeta function  can be analytically extended to the whole complex plane except by $s=1$ (unique pole). Note that the zeta function can be recover setting $\alpha=1$:
\[
\zeta(s,1)= \zeta(s).
\]
Let us recall the  Hermite integral representation of $\zeta(s,\alpha)$, valid for $\alpha\in (0,1]$.
\begin{lemma}{\cite[eqn. (10)]{Ada07}}
 One has
 \begin{equation}\label{zetahext}
 \zeta(s,\alpha)=\dfrac{1}{2\alpha^s}+\dfrac{\alpha^{1-s}}{s-1}+2\int_0^\infty\dfrac{\sin(s\arctan{\left(\frac{t}{\alpha} \right)})}{(\alpha^2+t^2)^{\frac{s}{2}}(e^{2\pi t}-1)}dt,
 \end{equation}
 for any  $ s\in \Com\backslash\{1\}$, and $\alpha\in (0,1]$.
\end{lemma}

\begin{remark}
Notice that  the right hand \eqref{zetahext} is a well-defined  for any $\alpha\in \Com$ such that $\re(\alpha)>0$ \cite[eqn. (10)]{Ada07}, but this will not be used in this paper.
\end{remark}

\subsection{Dirichlet characters and Dirichlet $L$-functions}%\label{sect:22}

In this subsection we introduce elements of analytic number theory needed through this work. As usual, $\hbox{gcd}(m,n)$ denotes the greatest common divisor between $m$ and $n$ of non negative integers. In case  
$\hbox{gcd}(m,n)=1$ we say that $m$ and $n$ are coprime. For the purposes of this work, we will only consider as input the set $\N$, but the following definitions are also available in the case of the integers $\Z$.

\begin{definition}\label{def:Dirichlet}
   A Dirichlet character with period $m\in\N$, $m>0$, is a complex-valued function $\chi_{m}:\N\to \Com$ satisfying:
    \begin{enumerate}
    \item Periodicity: $\chi_{m}(n+m)=\chi_{m}(n)$,\quad for all $n\in \N$.
        \item $\chi_{m}(ab)=\chi_{m}(a)\chi_{m}(b)$\quad for all  $ a,b\in \N$.   
        \item For all $n\in \N$,
       \begin{align*}
         &\chi_{m}(n)\neq 0 \quad \text{ if } \quad {gcd}(n,m)=1,\\
        &\chi_{m}(n)=0 \quad  \text{ if } \quad  gcd(n,m)>1. 
        \end{align*}
    \end{enumerate}
\end{definition}

\begin{remark}\label{propiedades_chi}
One easily gets that $\chi_m(1)=1$, $\chi_m(n^p)=\chi_m(n)^p$, $p\in\N$, and if $n_1\equiv n_2$ (mod $m$), then $\chi_m(n_1)=\chi_m(n_2)$. Additionally, if $m>1$, one always has $\chi_m(km)=0$ for any $k\in \{1,2,\ldots\}$. Finally, if $m=1$, $\chi_m(n)=1$ for all $n\in \N$.
\end{remark}

\begin{definition}\label{def:Dirichletprincipal}
   A Dirichlet character with period $m$ is called principal if satisfies the following conditions:  For all $n\in \N$
       \begin{align*}
         &\chi_{m}(n)=1 \quad \text{ if } \quad gcd(n,m)=1,\\
        &\chi_{m}(n)=0 \quad  \text{ if } \quad  gcd(n,m)>1. 
        \end{align*}
\end{definition}

A principal character acts as an identity in the group of characters: if $\chi_{m}$ is a principal character, and $\widetilde{\chi}_{m}$ is a general character, then one has $\widetilde{\chi}_{m}(n) \chi_n(n) =\chi_m(n)\widetilde{\chi}_{m} = \widetilde{\chi}_{m}(n)$ (notice that the product of two characters of the same modulus is another character of the same modulus).

\medskip

Let us recall the definition of the Dirichlet $L$-functions, previously introduced in \eqref{L_m_serie}. Let $\chi_{m}$ be a Dirichlet character of period $m$. The associated Dirichlet $L$-function is defined for $\re(s)>1$ as
\[%\begin{equation}\label{L_m_serie}
L_{m}(s)=\displaystyle{\sum_{n=1}^\infty \dfrac{\chi_{m}(n)}{n^s}}.
\]%end{equation}
As in the case of the $\zeta$ function, Dirichlet $L$-functions have an associated critical line, and a corresponding unique extension to $\re(s)<1$, which will be either meromorphic or analytic on $\Com$ depending on the character. Indeed, if the character $m$ is principal, then $L_{m}$ has a pole at $s=1$, and in the non principal case, it does not.  

\medskip

Dirichlet $L$-functions are deeply related to general zeta functions. Indeed, one has the following representation formula:
\begin{lemma}\label{lema:formulalsum}
Let $s\in \Com$, and $m$ a Dirichlet character. Then 
\begin{equation}\label{formula_Lm}
L_{m}(s)=\frac1{m^{s}}\sum_{r=1}^m\chi_{m}(r)\zeta\left(s,\dfrac{r}{m}\right).
\end{equation}
\end{lemma}
\begin{proof}
See \cite[p. 249]{Apo76}.
\end{proof}

Notice that \eqref{formula_Lm} is valid in the whole complex plane, provided $L_{m}(s)$ is well-defined, and does not use in principle the standard series representation of $L_m$. Indeed, 
\eqref{formula_Lm} consists only of a finite number of terms written in terms of the more involved general zeta $\zeta\left(s,\dfrac{r}{m}\right)$ introduced in \eqref{zeta_general}.

\section{Comparison principles}\label{sec:2b}

 In what follows, we recall invariant-manifold results stated in \cite{GNHY}, see also \cite{Wein}. To start with, consider the $2\times2$ ODE system % \eqref{eqn:edog}:
\begin{equation}\label{eqn:edog}
\begin{aligned}
\dot U(t)= &~{} F(U(t)), \quad t\ge 0,\\
U(0) = &~{} U_0 \in\R^2 \quad \hbox{given}.
\end{aligned}
\end{equation}
Here, $F:\R^2 \to \R^2$, $F=(F_1,F_2)$, is assumed as a locally Lipschitz function in oder to ensure at least local existence in time. Assume that $U$ is defined in the interval $[0,T]$. Consider also the Heat flow $2\times 2$ system associated to the previous model: %\eqref{eqn:edpg}
    \begin{equation}\label{eqn:edpg}
\begin{aligned}
\partial_t u (t,x)= &~{}\Delta u (t,x)+ F(u(t,x)), \quad x\in\R^d, \quad t\ge 0,\\
u(0,x) = &~{} (u_{0,1}(x),u_{0,2}(x)) = g(x) \quad \hbox{given}.
\end{aligned}
\end{equation}
In order to clarify $F$, later we shall use 
\[
F=  L_m , \quad F_1= \re L_m, \quad \quad F_2= \ima L_m.
\]
Let $u=u(t,x)\in\R^2$ be the solution of the equation \eqref{eqn:edpg}, and by taking $T>0$ smaller if necessary, consider $u(t,x)$ defined for $t\in[0,T]$. Let $\mathcal I[u](t)$ be its image at time $t$% \in[0,T]$ %we can define:
\[
\mathcal I[u](t):=\left\{ u(t,x)\in \R^2 ~ | ~ x\in \R^d\right\}, \quad t\in[0,T].
\]

\subsection{Invariant domains} By a classical argument, if the solution $U=U(t)$ of \eqref{eqn:edog} stays bounded during the maximal interval of existence $[0,T]$, then $T=+\infty$ and $U(t)$ exists globally. The following definition is precisely used to ensure that a priori estimates on the solutions are enough to obtain global existence in parabolic models.

\begin{definition}\label{Domaint}
Let $D(t)\subseteq \R^2$ be a domain parametrized by $t\geq 0$, bounded at each finite time $t$. We say that $D=(D(t))_{t\geq 0}$ is invariant under the flow $U(t)$ of \eqref{eqn:edog} if $U_0 \in D(0)$ implies $U(t)\in D(t)$, for all $t>0$.
\end{definition}

Even if the previous definition seems to require the global existence of $U(t)$, it turns out that it is easily concluded that if $D=(D(t))_{t\geq 0}$ is invariant under the flow $U(t)$ for any time $[0,T]$, then it is global. This subtle property will be assumed in several parts of this paper, but always checking first that the required invariant domains are locally bounded in time.

\medskip

Invariant sets are key in the study of long time properties of diffusive models, as shown in the following well-known result.

\begin{lemma}[Weinberger \cite{Wein}, see also Lemma 2.1 in \cite{GNHY}]\label{lem:lema1}
 Suppose that $D(t) \subset \R^2$ is convex for each $t > 0$ and, $\{D(t)\}_{t>0}$ is
 invariant under the flow $U(t)$ in \eqref{eqn:edog}. If $\mathcal I[u](t_0) \subseteq D(t_0)$ for some $t_0 > 0$, then $\mathcal I[u](t)\subseteq D(t)$ for all $t > t_0.$
\end{lemma}

Notice that this lemma states that the property of invariance under the flow $U(t)$ (Definition \ref{Domaint}) is transferred to the PDE \eqref{eqn:edpg}.  Additional to the previous result, we shall need the following invariant subspace lemma.

\begin{lemma}[Lemma 2.2 in \cite{GNHY}]\label{lem:lema2}
Consider a finite set of $M$ functions of the form $\{H_j=H_j(t,u_1,u_2)\}_{1\leq j \leq M}$, where each $H_j$ is a $C^1$ function from $\R^3$ into $\R$. Suppose additionally that $D(t)$ is expressed as
 \begin{equation}\label{D_inter}
 D(t)=\bigcap_{j=1}^M\left\{(u_1,u_2) \in \R^2 : H_j(t,u_1,u_2) < 0\right\},\quad t > 0.
 \end{equation}
If 
\[ 
\dfrac{d}{dt}H_j(t,U_1(t),U_2(t)) \leq 0\quad on \quad \left\{(u_1,u_2) \in \partial D(t) : H_j(t,u_1,u_2) = 0\right\},
\]
 and for all $j = 1,2,...,M.$
Then $\{D(t)\}_{t>0}$ is  invariant under the flow \eqref{eqn:edog}.
\end{lemma}

Notice that the set \eqref{D_inter} can be expressed in a simpler form as
\[
D(t)=\bigcap_{j=1}^M (H_j(t,\cdot,\cdot))^{-1}((-\infty,0)),
\]
which in topological terms represents a sort of open set of the related inverse topology induced by the family of functions $H_j$.
%{\color{blue} Entiendo la idea de la otra expresión, pero no creo que este bien, ya que el parametro t queda libre, pues la pre imagen nos entragaria algo de 
%3 componentes, creo que hay un detalle que modificar como fijando el tiempo}

%\begin{proof}
%    These Lemmas are enunciate in \cite{GNHY} and 
%    proved in \cite{Wein}.
%\end{proof}
 
\subsection{Comparison Principles} The subsequent result is a classical lemma that provides  pointwise bounds of solutions to nonlinear heat systems, presented within a broader and more general context.
%Let $F :\mathbb R^2 \to \mathbb R^2$ and $G:\R \to \R$ be fixed and locally Lipschitz functions, where $F=(F_1,F_2)$. As in \eqref{eqn:edpg}, for some small $T>0$ to be determined later, we consider the PDE system for $u=(u_1,u_2)$:
%  \begin{equation}\label{eqn:edpg}
%    \begin{aligned}
%        \partial_t u(t,x)&=\Delta u(t,x)+F(u(t,x)), \quad (t,x)\in [0,T]\times \R^d,\\
%       % \partial_t u_2(t,x)&=\Delta u_2(t,x)+F_2(u(t,x)),\\ %\\  \quad \forall (t,x)\in [0,T]\times \R^d\notag\\
%        u(0,x)&=(u_{0,1}(x),u_{0,2}(x))\quad  x\in \R^d.
%        \end{aligned}
%   \end{equation}
%Let us also consider the equivalent ODE system for $U=(U_1,U_2)$ (compare with \eqref{eqn:edog})
%\begin{equation}\label{eqn:edog}
%\begin{aligned}
%\dot U(t)= &~{} F(U(t)), \quad t\ge 0,\\
%U(0) = &~{} U_0 \in\R^2 \quad \hbox{given}.
%\end{aligned}
%\end{equation}  
%This system has a unique solution, by making $T$ smaller if necessary.  

\medskip

Let $G:\R \to \R$ be fixed and locally Lipschitz function.  $G$ will describe the nonlinear term in a scalar ODE system. Indeed, let $V_0\in\R$ and $\varepsilon\in \R$. Consider the scalar ODE given by
 \begin{equation}\label{heat2}
    \begin{aligned}
        \dot V(t,V_0,\varepsilon)=G(V(t,V_0,\varepsilon)), \quad V(0,V_0,\varepsilon)=V_0+\varepsilon, \quad \varepsilon\in\mathbb R.
    \end{aligned}
    \end{equation}
Let $V$ be a solution defined in the internal of existence $[0,T_\varepsilon)$, with $T_\varepsilon\leq +\infty$ maximal. We shall assume that for all  $- \varepsilon_0 <\varepsilon <\varepsilon_0$ one has $0<T_\varepsilon\leq +\infty$, and that 
\[
T_*:= \inf_{-\varepsilon_0 < \varepsilon <\varepsilon_0} T_\varepsilon  >0.
\]
Later we will verify that this is precisely the case. In what follows, by taking $T$ smaller if necessary, we assume $t<T\leq T_*$.

%{\color{red}ac'a el tiempo de existencia $T_\varepsilon$ depende de epsilon, es necesario epsilon?}{\color{blue} Si para poder tomar las desigualdades con mayor o igual y no estricto. Pero podriamos tomar el minimo de los tiempos de existencia en un rango entre $-\varepsilon_0,\varepsilon_0$.}

\begin{lemma}[Comparison principle]\label{comparisonprinciple}
   Under the previous assumptions on $T$, the following is satisfied. There exists $\varepsilon_0>0$ such that, for all $0<\varepsilon <\varepsilon_0$ the solutions $u$ and $V(\cdot,V_0,\varepsilon)$ of \eqref{eqn:edpg}  and \eqref{heat2}, respectively,  satisfy
    \medskip
    \begin{enumerate}
    \item If $F_1(V(t,V_0,\varepsilon),u_2)\leq G(V(t,V_0,\varepsilon))$ for all $t\geq 0$, $u_2\in\R$ and $\varepsilon>0$, and $u_{0,1}(x)\leq V_{0}$ for all $x\in\R^d$, then one has
    \[
    u_1(t,x)\leq V(t,V_0,0),\quad \text{for all } (t,x)\in[0,T]\times \R^d.
    \]
\item If now $F_1(V(t,V_0,\varepsilon),u_2)\geq G(V(t,V_0,\varepsilon))$ for all $t\geq 0$, $u_2\in\R$ and $\varepsilon<0$, and $u_{0,1}(x)\geq V_0$  for all $x\in\R^d$, then one has
    \[
    u_1(t,x)\geq V(t,V_0,0),\quad \text{for all } (t,x)\in[0,T]\times \R^d.
    \]
   \end{enumerate} 
\end{lemma}

\begin{remark} Notice that the conclusions in the previous lemma are only stated for the interval of time $[0,T]$, this is required since we still do not know whether or not $u(t)$ is globally defined. In any case, later we shall prove that under a particular condition on the initial data $g$, the solution is global and we can extend the previous lemma to any finite time $t$.
\end{remark}

\begin{proof}[Proof of Lemma \ref{comparisonprinciple}]
Let $H(t,u_1,u_2) : = u_1- V(t,V_0,\varepsilon)$, where $V$ solves \eqref{heat2}. Let $U$ be the  solution of \eqref{eqn:edog}. Firstly  note that
\begin{equation}\label{raro0}
\dfrac{d}{dt}H(t,U_1(t),U_2(t))=\dot{U}_1(t)-\dot V(t,V_0,\varepsilon) = F_1(U(t))-G(V(t,V_0,\varepsilon)).
\end{equation}
Secondly, we consider for each $t\in [0,T]$ the set
\[
D(t):=\left\{(u_1,u_2)\in \R^2 ~{} \Big| ~{} H(t,u_1,u_2)<0\right\}.
\]
Notice that in $D(t)$, $t\in [0,T]$, one has
\[
U_1(t) < V(t,V_0,\varepsilon)<+\infty.
\]
%{\color{red} ojo que no hemos asegurado que $U_2(t)$ es finito.}
For each time $t$ we have that $(U_1(t),U_2(t))\in \partial D(t)$ is equivalent to $U_1(t)=V(t,V_0,\varepsilon)$. In this boundary set one has from \eqref{raro0} and the hypothesis, 
\begin{equation}\label{raro}
\begin{aligned}
\dfrac{d}{dt}H(t,U_1(t),U_2(t))= &~{} F_1(U(t))-G(V(t,V_0,\varepsilon)) \\
=&~{} F_1(V(t,V_0,\varepsilon),U_2(t))  -G(V(t,V_0,\varepsilon))\leq 0.
\end{aligned}
\end{equation}
Since \eqref{raro} is satisfied, Lemma \ref{lem:lema2} ensures that $(D(t))_{t\in [0,T]}$ is an invariant set. Therefore  since,  $u_{1,0}(x)\leq V_0 <V(0,V_0,\varepsilon)$ is satisfied,   Lemma \ref{lem:lema1} implies  
\[
H(0,u_{0,1}, u_{0,2})=u_{0,1} - V(0,V_0,\varepsilon)<0.
\]  
Consequently $u_1(t,x)\in D(t)$, or equivalently, $u_1(t,x)< V(t,V_0,\varepsilon)$, at least in the interval of time $[0,T]$. Finally using that $\varepsilon>0$ is arbitrarily small one has 
\[
u_1(t,x)\leq V(t,V_0,0).
\]
This ends the proof of (1). The case (2) is analogous; therefore, we omit its proof.
%\medskip
%{\color{blue}Felipe: tengo la misma duda del tiempo $T_\varepsilon$ pues al mover $\varepsilon$ a cero te puedes quedar sin intervalo de tiempo}
\end{proof}

As a direct consequence of Lemma \ref{comparisonprinciple} the  following result is valid.

\begin{cor}\label{comparisonprinciple2}
Let $I\leq S$ be two fixed real numbers. Let $F=(F_1,F_2)$ be a locally Lipschitz vector function. Let $V(t,I,0)$ and $V(t,S,0)$ be the solutions of \eqref{heat2} with initial conditions $I$ and $S$, respectively, and $G:=F_1$.  Let $u=u(t,x)$ be the solution of \eqref{eqn:edpg} with given initial condition $u(0,x)=g(x)$ and defined on an interval of time $[0,T]$.  By restricting $T$ if necessary, $u(t)$, $V(t,I,0)$ and $V(t,S,0)$ are all assumed being defined on $[0,T]$. Then the following are satisfied:

\begin{enumerate}
\item[$(i)$] If $I\leq u_{0,1}(x)$ for all $x\in\R^d$,  one has
    \[
    V(t,I,0)\leq u_1(t,x),\quad \text{for all } (t,x)\in[0,T]\times \R^d.
    \]   
\item[$(ii)$]    If $u_{0,1}(x)\leq S$ for all $x\in\R^d$,  one has
    \[
    u_1(t,x)\leq V(t,S,0),\quad \text{for all } (t,x)\in[0,T]\times \R^d.
    \]   
\end{enumerate}    
\end{cor}
\begin{proof}
Direct from the Lemma \ref{comparisonprinciple} by noticing that $F_1(V(t,V_0,\varepsilon),u_2) = G(V(t,V_0,\varepsilon))$.
\end{proof}
 
We finish this section with the following result, direct consequence of Lemma \ref{comparisonprinciple}:
\begin{cor}\label{comparisonprinciple3}
  Let $u=u(t,x)$ be a solution of \eqref{eqn:edpg} defined in $[0,T]$. Let $M_1\leq M_2$ be given such that for some $j\in\{1,2\}$,
  \[
    I\leq u_j(t,x)\leq S, \quad \hbox{ and } \quad M_1\leq F_j(u)(t,x)\leq M_2.
   \]
  Then one has%we have% we have that:
    \[
     I +M_1 t\leq u_j(t,x)\leq S+M_2 t,\quad \hbox{for all} \quad (t,x)\in[0,T]\times \R^d.
    \]   
\end{cor}

\begin{proof}
 Let $V_{M_1}(t,I,0)$, $V_{M_2}(t,S,0)$ be solutions to \eqref{heat2} with $G\equiv M_1$ and $G\equiv M_2$, respectively. One has $V_{M_1}(t,I,0) = I +M_1 t$ and $V_{M_1}(t,I,0) = S +M_2 t$. From Corollary \ref{comparisonprinciple2} we conclude.
\end{proof}

\section{Uniform bounds}\label{sect:3}

The purpose of this section is to exhibit  key  bounds of  the zeta function and some  generalization to other class of functions. These bounds  will be important in the proof of our main local well-posedness result Theorem \ref{LDirichlet}. 
\begin{definition}    
Let $\beta,\varepsilon>0$ and $P$ as in Definition \ref{def:P}. We define the closed set $\mathcal B= \mathcal B_{\varepsilon, \beta, T}$ as follows:
\begin{equation}\label{B}
\mathcal B:=\left\{u\in X \Big| \text{ such that }\|u\|_X\leq \beta \quad \hbox{and}\quad  \inf_{x\in \R^d}P(u(t))\geq \varepsilon, ~ \text{for all } t\in [0,T]\right\}.
\end{equation}
See Fig. \ref{mB} for further details.
\end{definition}

\begin{figure}[H]
\tikzset{every picture/.style={line width=0.75pt}} %set default line width to 0.75pt        
    \begin{center}
        \tikzset{every picture/.style={line width=0.75pt}} %set default line width to 0.75pt                
        \begin{tikzpicture}[x=0.75pt,y=0.75pt,yscale=-0.55,xscale=0.55,scale=0.85]
        
%Shape: Circle [id:dp4464167410317119] 
\draw  [fill={rgb, 255:red, 255; green, 255; blue, 255 }  ,fill opacity=0.55 ] (334,248.6) .. controls (334,239.21) and (341.61,231.6) .. (351,231.6) .. controls (360.39,231.6) and (368,239.21) .. (368,248.6) .. controls (368,257.99) and (360.39,265.6) .. (351,265.6) .. controls (341.61,265.6) and (334,257.99) .. (334,248.6) -- cycle ;
%Straight Lines [id:da02130453848814029] 
\draw [color={rgb, 255:red, 208; green, 2; blue, 27 }  ,draw opacity=1 ]   (351,250) -- (351,234.6) ;
\draw [shift={(351,232.6)}, rotate = 90] [color={rgb, 255:red, 208; green, 2; blue, 27 }  ,draw opacity=1 ][line width=0.75]    (10.93,-3.29) .. controls (6.95,-1.4) and (3.31,-0.3) .. (0,0) .. controls (3.31,0.3) and (6.95,1.4) .. (10.93,3.29)   ;
%Shape: Square [id:dp6928440450151296] 
\draw   (150.2,101) -- (450,101) -- (450,400.8) -- (150.2,400.8) -- cycle ;
%Curve Lines [id:da5567193434664772] 
\draw [color={rgb, 255:red, 155; green, 155; blue, 155 }  ,draw opacity=1 ]   (239,292.6) .. controls (215,279.6) and (204.93,232.11) .. (228,210.6) .. controls (251.07,189.09) and (271,179.6) .. (280,198.6) .. controls (289,217.6) and (256,286.6) .. (297,292.6) .. controls (338,298.6) and (324,179.6) .. (340,175.6) .. controls (356,171.6) and (380.01,235.6) .. (372,268.6) .. controls (363.99,301.6) and (353,302.6) .. (347,296.6) .. controls (341,290.6) and (324.81,266.62) .. (310.87,243.68) .. controls (296.94,220.73) and (224,220.6) .. (237,241.6) ;
%Straight Lines [id:da07756167204598641] 
\draw    (500,99) -- (500,401.4) ;
\draw [shift={(500,401.4)}, rotate = 270] [color={rgb, 255:red, 0; green, 0; blue, 0 }  ][line width=0.75]    (0,5.59) -- (0,-5.59)   ;
\draw [shift={(500,99)}, rotate = 270] [color={rgb, 255:red, 0; green, 0; blue, 0 }  ][line width=0.75]    (0,5.59) -- (0,-5.59)   ;
%Straight Lines [id:da06935528652735856] 
\draw    (49,249.8) -- (599,249.8) ;
\draw [shift={(602,249.8)}, rotate = 180] [fill={rgb, 255:red, 0; green, 0; blue, 0 }  ][line width=0.08]  [draw opacity=0] (8.93,-4.29) -- (0,0) -- (8.93,4.29) -- cycle    ;
%Straight Lines [id:da9223280768992612] 
\draw    (301,429.8) -- (301,51.8) ;
\draw [shift={(301,48.8)}, rotate = 90] [fill={rgb, 255:red, 0; green, 0; blue, 0 }  ][line width=0.08]  [draw opacity=0] (8.93,-4.29) -- (0,0) -- (8.93,4.29) -- cycle    ;

% Text Node
% Text Node
\draw (505,190.4) node [anchor=north west][inner sep=0.75pt]   {$2\beta $};
% Text Node
\draw (379,213.4) node [anchor=north west][inner sep=0.75pt]    {$\textcolor[rgb]{0.82,0.01,0.11}{\varepsilon }$};
% Text Node
\draw (571,210) node [anchor=north west][inner sep=0.75pt]    {$\re$};
% Text Node
\draw (221,145.4) node [anchor=north west][inner sep=0.75pt]  [color={rgb, 255:red, 155; green, 155; blue, 155 }  ,opacity=1 ]  {$\mathcal{I}[ g]$};
% Text Node
\draw (344,268.4) node [anchor=north west][inner sep=0.75pt]    {$1$};
% Text Node
\draw (160,107.4) node [anchor=north west][inner sep=0.75pt]    {$B_{\beta ,\varepsilon }$};
% Text Node
\draw (240,52.4) node [anchor=north west][inner sep=0.75pt]    {$\ima$};

\end{tikzpicture}

        \end{center}
        \caption{The set $\mathcal B$ defined in \eqref{B}, notice that $g$ must be far from the pole and bounded by a value $\beta$.}\label{mB}
\end{figure}
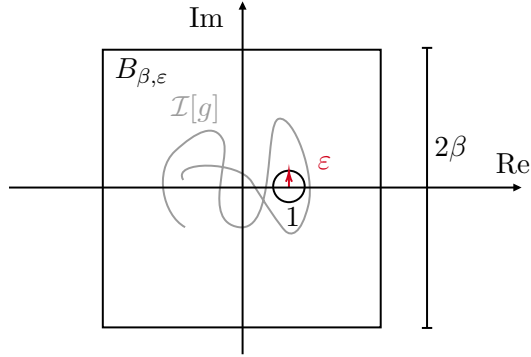
In the next $[-\beta,\beta]^2$ denotes  the closed square in the complex plane with length $2\beta>0$. We start  with a classical estimate for analytic functions. 

 \begin{lemma}\label{lema:lipzholo2}
  Let $h=h(s)$ be an analytic function on $\Com$. Then for all $u,v\in \mathcal B$, one has
  \begin{equation}\label{lema:lipzholo2_b}
  \|h(u)-h(v)\|_X\leq \sqrt{2} \left( \sup_{z\in [-\beta,\beta]^2}  |h'(z)| \right) \|u-v\|_X.
  \end{equation}
 \end{lemma}
 
 \begin{remark}
 Notice that \eqref{lema:lipzholo2_b} is an estimate that only considers functions $u,v\in B$, and locally, the region $[-\beta,\beta]^2$. In that sense, it is an estimate that do not avoid the problem of the zeta, a pole at $s=1$. Later, this estimate will be improved.
 \end{remark}
 
 \begin{proof}[Proof of Lemma \ref{lema:lipzholo2}]
 For all $s,p\in [-\beta,\beta]^2$, $s=s_1+is_2$, $p=p_1+ip_2$, there exist $\xi \in [-\beta,\beta]^2$ such that
\[
|h_1(s)-h_1(p)|=   |\nabla h_1(\xi)\cdot (s-p)|.
\]
From the Cauchy-Schwarz inequality, 
\[
|h_1(s)-h_1(p)|\leq \sqrt{h_{1x}(\xi)^2+h_{1y}(\xi)^2}\sqrt{(s_1-p_1)^2+(s_2-p_2)^2}.
\]
By the Cauchy-Riemann conditions, we obtain
\[
|h_1(s)-h_1(p)|\leq \sqrt{h_{1x}(\xi)^2+h_{2x}(\xi)^2}\sqrt{(s_1-p_1)^2+(s_2-p_2)^2},
\]
hence,
\[
|h_1(s)-h_1(p)|\leq |h'(\xi)|\sqrt{(s_1-p_1)^2+(s_2-p_2)^2},
\]
and
\[
|h_1(s)-h_1(p)|\leq \sqrt{2} |h'(\xi)|\max_{i\in\{1,2\}}\{|s_i-p_i|\}.
\]
Similarly, there exists $\tilde\xi\in [-\beta,\beta]^2$ such that
\[
|h_2(s)-h_2(p)|\leq \sqrt{2} |h'(\tilde\xi)|\max_{i\in\{1,2\}}\{|s_i-p_i|\}.
\]
Therefore 
\[
\max_{i\in\{1,2\}}\{|h_i(s)-h_i(p)|\}\leq  \sqrt{2}  \max\{|h'(\xi)|,|h'(\tilde\xi)|\}\max_{i\in\{1,2\}}\{|s_i-p_i|\}.
\]
Let $I= (0,T]$. If $u,v\in \mathcal B$, then for each $x\in\mathbb R^d$ one has that $\{u(t,x)\}_{t\in I},\;\{v(t,x)\}_{t\in I}\subseteq [-\beta,\beta]^2$ are included in a convex subset of the complex plane. Therefore there are $\xi_{t,x},\tilde \xi_{t,x}\in [-\beta,\beta]^2$ such that
\[
\begin{aligned}
& \max_{i\in\{1,2\}}\{|h_i(u(t,x))-h_i(v(t,x))|\} \\
&~{} \quad \leq \sqrt{2}\max\{|h'(\xi_{t,x})|,|h'(\tilde \xi_{t,x})|\}\max_{i\in\{1,2\}} |u_i(t,x)-v_i(t,x)| .
\end{aligned}
\]
Since $\xi_{t,x}, \tilde\xi_{t,x}\in [-\beta,\beta]^2$ we obtain
\[
\max_{i\in\{1,2\}} |h_i(u(t,x))-h_i(v(t,x))| \leq \sqrt{2} \left( \sup_{z\in [-\beta,\beta]^2}  |h'(z)| \right) \max_{i\in\{1,2\}} |u_i(t,x)-v_i(t,x)| .
\]
Taking supremum norm on $x$ and $t$ respectively, we get
\[
\|h(u)-h(v)\|_{X}\leq \sqrt{2} \left( \sup_{z\in [-\beta,\beta]^2}  |h'(z)| \right) \|u-v\|_X,
\]
as desired.
\end{proof}

\subsection{The Hurwitz zeta case} Recall the Hurwitz's zeta function $\zeta(s,\alpha)$ introduced in \eqref{zeta_general}.  Since $\zeta(s,\alpha)$ has a unique simple pole at $s=1$ (with residue $1$) from equation \eqref{zetahext} we obtain the following decomposition
 \begin{equation}\label{deco_zeta_h}
 \zeta(s,\alpha)=\dfrac{1}{s-1}+d(s,\alpha)+h(s,\alpha),
 \end{equation}
 where
 \begin{equation}\label{d}
 d(s,\alpha) : =\dfrac{\alpha^{1-s}-1}{s-1}+\dfrac{1}{2\alpha^s},
\end{equation}
and
\begin{equation}\label{h}
h(s,\alpha):=2\int_0^\infty\dfrac{\sin(s\arctan{\frac{t}{\alpha}})}{(\alpha^2+t^2)^{\frac{s}{2}}(e^{2\pi t}-1)}dt.
\end{equation}
Note that $d$ represents an entire function, since  $d(s,\alpha)$ can be analytically extended to the full complex plane. In the next lines, $d$ and $h$ will denote the functions in \eqref{d} and \eqref{h}. % describe as above.

\medskip

We will prove some explicit bounds on $d$ and $h$, and its derivatives. These estimates are better than standard ones obtained by invoking the fact that one works with continuous functions on compact subsets of the complex plane, because the better estimates one gets from the nonlinearity, the more precise local and global well-posedness results one obtains. 

\medskip

The following two lemmas are of technical and will be proved in Appendix \ref{app:A} and \ref{app:B}, respectively:

\begin{lemma}[Explicit bounds on $h$ and $h'$]\label{lemma:boundh}
Let $h$ be as in \eqref{h}. There exist $H_{1,\alpha,\beta},H_{2,\alpha,\beta}>0$, only depending on $\alpha$ and $\beta$, and such that for any $s \in [-\beta,\beta]^2$,
    \[
\left| h|_{[-\beta,\beta]^2}(s,\alpha) \right|\leq H_{1,\alpha,\beta} \quad \text{and}\quad   \left| h' |_{[-\beta,\beta]^2}(s,\alpha)\right|\leq H_{2,\alpha,\beta}.
    \]
\end{lemma}

\begin{lemma}[Explicit bounds on $d$ and $d'$]\label{lemma:boundd}
 Let $d$ be as in \eqref{d}. Then there exist $D_{1,\alpha,\beta},D_{2,\alpha,\beta}>0$ only depending on $\alpha$ and $\beta$, such that
\[
\|d|_{[-\beta,\beta]^2}(\cdot,\alpha)\|_Y\leq D_{1,\alpha,\beta} \quad \text{and} \quad \|d'|_{[-\beta,\beta]^2}(\cdot,\alpha)\|_Y\leq D_{2,\alpha,\beta}.
\]
\end{lemma}

The next lemma establish  estimates of the Hurwitz's zeta function:
\begin{lemma}\label{lemma:boundZ}
 Let $u,v\in \mathcal B$. Then there exists $Z_{1,\alpha,\beta,\varepsilon}, Z_{2,\alpha,\beta,\varepsilon}>0$, only depending on $\alpha, \beta$ and $\varepsilon$, and such that 
\[
\|  \zeta(u,\alpha)\|_X\leq Z_{1,\alpha,\beta,\varepsilon}\quad \text{and}\quad
\|  \zeta(u,\alpha)-\zeta(v,\alpha)\|_X\leq Z_{2,\alpha,\beta,\varepsilon}\|u-v\|_X.
\]
\end{lemma}

\begin{proof} 

Let $h$ and $d$ be described as in \eqref{h} and \eqref{d}. From Lemmas \ref{lema:lipzholo2_b} and \ref{lemma:boundh} we have that
\[
\begin{aligned}
\|h(u,\alpha)-h(v,\alpha)\|_{X}\leq & \sqrt{2}\|h'|_{[-\beta,\beta]^2}( \cdot ,\alpha)\|_{Y}\|u-v\|_X \\
\leq & \sqrt{2}H_{2,\alpha,\beta}\|u-v\|_X.
\end{aligned}
\]
Also from  Lemma \ref{lemma:boundd} we obtain
\[
\begin{aligned}
\|d(u,\alpha)-d(v,\alpha)\|_{X}\leq & \sqrt{2}\|d'|_{[-\beta,\beta]^2}(\cdot ,\alpha)\|_{Y}\|u-v\|_X \\
\leq &~{}\sqrt{2} D_{2,\alpha,\beta}\|u-v\|_X.
\end{aligned}
\]
On the other hand,  for all $u\in \mathcal B$ we know that 
$\varepsilon\leq \|u-1\|_X$ implies $\left\|\frac{1}{u-1}\right\|_X\leq \frac{1}{\varepsilon}.$
Thus 
\[
\left\|\dfrac{1}{u-1}-\dfrac{1}{v-1}\right\|_X=\left \|\dfrac{v-u}{(u-1)(v-1)}\right\|_X\leq\dfrac{1}{\varepsilon^2}\|v-u\|_X.
\]
Therefore, we get from \eqref{deco_zeta_h}, Lemmas \ref{lemma:boundd} and \ref{lemma:boundh},
\[
\begin{aligned}    
\|\zeta(u,\alpha)\|_X\leq &~{} \dfrac{1}{\varepsilon}+H_{1,\alpha,\beta}+D_{1,\alpha,\beta}=: Z_{1,\alpha,\beta,\varepsilon}.
\end{aligned}
\]
Finally,
\[
\begin{aligned}
\| \zeta(u,\alpha)-\zeta(v,\alpha)\|_{X} \leq  &~{} \dfrac{1}{\varepsilon^2} \|u-v\|_{X} +\sqrt{2}H_{2,\alpha,\beta}\|u-v\|_X+\sqrt{2}D_{2,\alpha,\beta}\|u-v\|_X \\
= :&~{} Z_{2,\alpha,\beta,\varepsilon}\|u-v\|_{X}.
\end{aligned}
\]
The proof is complete.
\end{proof}
\begin{remark}
    Using that all auxiliary constants are decreasing in $\alpha$, we have that the constants $Z_{2,\alpha,\beta,\varepsilon}$ and $Z_{1,\alpha,\beta,\varepsilon}$ are decreasing in $\alpha$.
\end{remark}

\subsection{The general Dirichlet $L$-function case} Lemma \ref{lemma:boundZ} can be extended to  general Dirichlet zeta functions. Indeed, thanks to Lemma \ref{lema:formulalsum}, $L_m$ can be written as a linear combination of Hurwitz zeta functions. 

\begin{lemma}\label{lemma:boundL}
 Let $u,v\in B$. Then there exist $M_{1,\beta,\varepsilon,m},M_{2,\beta,\varepsilon,m}>0$, only depending on $\beta, \varepsilon$ and $m$, such that:
\[
\|L_{m}(u)\|_X\leq M_{1,\beta,\varepsilon,m}, \qquad \text{and}\qquad 
\|L_{m}(u)-L_{m}(v)\|_X\leq M_{2,\beta,\varepsilon,m}\|u-v\|_X.
\]
\end{lemma}

\begin{proof}
Let $g_1(s)$ and $g_2(s)$ be defined such that
\[
g_1(s) :=m^{-s} \quad \text{ and } \quad g_2(s):=\sum_{r=1}^m\chi_{m}(r)\zeta\left(u,\frac{r}{m}\right). 
\]
Therefore, from \eqref{formula_Lm} one has $L_m(s)=g_1(s)g_2(s)$, and
\[
\|g_1(u)\|_X\leq m^{\beta}.
\]
Using Lemma \ref{lema:lipzholo2_b}, we have that $g_1(s)$ satisfies
\[
\|g_1(u)-g_1(v)\|_X\leq \sqrt{2}\ln(m)m^{\beta} \|u-v\|_X\leq m^{\beta+1}\|u-v\|_X.
\]
Since  the $Y$-norm of a Dirichlet character is less than $1$ and $Z_{1,\alpha,\beta,\varepsilon}$ in Lemma \ref{lemma:boundZ}  is decreasing in $\alpha$ we get 
\[\|g_2(u)\|_X\leq \sum_{r=1}^m \left\|\zeta \left(s,\frac{r}{m}\right)\right\|_X\leq mZ_{1,\frac{1}{m},\beta,\varepsilon}. \]
Similarly using that $Z_{2,\alpha,\beta,\varepsilon}$ in Lemma \ref{lemma:boundZ} is a decreasing function on 
$\alpha$,  we have 
\[
\begin{aligned}
 \left\|g_2(u)-g_2(v)\right\|_X &\leq \sum_{r=1}^m Z_{2,\frac{r}{m},\beta,\varepsilon} \|u-v \|_{X}\\
&\leq m Z_{2,\frac{1}{m},\beta,\varepsilon}\|u-v \|_{X}.
\end{aligned}
\]
Lemma \ref{lema:formulalsum} implies
\[
\|L_m(u)\|_X\leq \|g_1(u)\|_X\|g_2(u)\|_X\leq m^{\beta+1 }Z_{1,\frac{1}{m},\beta,\varepsilon} =:M_{1,\beta,\varepsilon,m}.
\]
Finally  using again Lemma \ref{lema:formulalsum} we get
\[\begin{aligned}
\|L_{m}(u)-L_{m}(v)\|_X\
 \leq &~{} \|g_1(u)g_2(u)-g_1(u)g_2(v)\|_X+\|g_1(u)g_2(v)-g_1(v)g_2(v)\|_X\\
 \leq &~{} \|g_1(u)\|_X\|g_2(u)-g_2(v)\|_X+\|g_2(u)\|_X\|g_1(u)-g_1(v)\|_X\\
 \leq &~{} m^{\beta+1}Z_{2,\frac{1}{m},\beta,\varepsilon}\|u-v\|_X+m^{\beta+2}Z_{2,\frac{1}{m},\beta,\varepsilon}\|u-v\|_X.\\
 =: &~{} M_{2,\beta,\varepsilon,m}\|u-v\|_X.
\end{aligned}
\]
The proof is complete.
\end{proof}

\section{Local well-posedness}\label{sect:4}

In this section we prove Theorem \ref{LDirichlet}. The idea is classical with only minor differences. For this purpose,  for each $g\in Y$, $g= g_1+ i g_2 $, $g_i \in \mathbb R$,  we suppose that
\begin{equation*}
P(g):=\inf_{x\in \R^d} \left( |g_1-1|+|g_2| \right)>0.
\end{equation*}
Note that  $g\in Y$ implies $\|g\|_Y<\infty$. Let  $\beta>0$ be such that $2\|g\|_Y=\beta$. Since $g$ satisfies that $P(g)>0$ there exists  $\varepsilon>0$ such that
\[
\inf_{x\in\R^d}P(g)= 3\varepsilon.
\]
Given $\beta$ and $\varepsilon$ as before, consider $\mathcal B$ as described in Definition \ref{B}. Now we focus our attention in the map  $F: \mathcal B\to X$ described by 
\[
F(u):=e^{t\Delta}g+\int_{0}^{t} e^{(t-s)\Delta} L_m (u)ds,
\]
where $e^{t\Delta}h=K_t\ast h$ and $K_t(x)=\dfrac{1}{(4\pi t)^{d/2}}e^{-\frac{|x|^2}{4t}}$ is the classical heat kernel.  Let $t\in[0,T]$, with $T>0$ to be defined later on. Note  that from \eqref{def:YX},
\begin{align*}
\|F(u(t))\|_{Y} & \leq \left\|K_t \ast g+\int_{0}^t K_{t-s}\ast L_{m}(u(s))ds\right\|_Y\\
& \leq \| K_t\ast g\|_Y+\int_{0}^t \left\|K_{t-s}\ast L_{m}(u(s))\right\|_Y ds.
\end{align*}
Since  $f\in L^1(\R^d)$ and $g\in Y$ imply that $\|f\ast g\|_{Y}\leq \|f\|_{L^1}\|g\|_Y $, one obtains
\begin{align*}
\|F(u(t))\|_{Y} & \leq \|K_t\|_{L^1}\|g\|_Y+\int_{0}^t \|K_{t-s}\|_{L^1}\|L_{m}(u(s))\|_Y ds.
\end{align*}
Lemma \ref{lemma:boundL} combined with $\|K_t\|_{L^1}=1$ allow us to conclude that 
\begin{align*}
\|F(u(t))\|_{Y} & \leq \|g\|_Y+TM_{1,\beta,\varepsilon,m}.
\end{align*}
Thus, setting $T\leq \dfrac{\beta}{2M_{1,\beta,\varepsilon,m}}$ we obtain $\|F(u(t))\|_{Y} \leq \|g\|_Y+\dfrac{\beta}{2}\leq \beta$. Therefore, taking supremum on $t\in[0,T]$, 
\[ 
\|F(u)\|_X\leq \beta.
\]
It remains to show that $\inf_{x\in \R^d}P(F(u)(t))\geq \varepsilon$, for all $t\in [0,T]$. Indeed,  since $\inf\limits_{x\in \R^d} P(g)\geq 3\varepsilon$ we have from \eqref{def:P} that for all $x\in\R^d$,
\begin{equation}\label{cases}
 |g_1(x)-1|\geq \dfrac{3\varepsilon}{2}, \quad \hbox{or} \quad  |g_2(x)|\geq \dfrac{3\varepsilon}{2}.
\end{equation}
 In the first case, by the continuity of  $g_1$, one as  for each $x\in\R^d$ that $g_1(x)\geq 1+\dfrac{3\varepsilon}{2}$ or  $g_1(x)\leq 1- \dfrac{3\varepsilon}{2}.$ Therefore,  due to the positivity of the kernel $K$ we conclude 
\[ 
(K_t\ast g_1)(x)\geq 1+\dfrac{3\varepsilon}{2} \quad \hbox{or} \quad (K_t\ast g_1)(x)\leq 1- \dfrac{3\varepsilon}{2}.
\]
Hence $  |e^{t\Delta }g_1(x)-1|\geq \dfrac{3\varepsilon}{2}$ and 
\begin{equation}\label{P1}
\inf\limits_{x\in \R^d}|e^{t\Delta }g_1(x)-1|\geq \dfrac{3\varepsilon}{2}.
\end{equation}
If now $ |g_2(x)|\geq \dfrac{3\varepsilon}{2}$ in \eqref{cases}, we have that
\begin{equation}\label{P2}
\inf_{x\in\R^d}|e^{t\Delta }g_2(x)|\geq \dfrac{3\varepsilon}{2}.
\end{equation}
The inequalities \eqref{P1} and \eqref{P2} imply
\begin{equation}\label{paso1}
P(e^{t\Delta }g)\geq \dfrac{3\varepsilon}{2}.
\end{equation}
Now  we estimate: 
\[
\begin{aligned}
P(F(u(t))) = &~{} \inf_{x\in \R^d} \left( | (F(u))_1-1|+|(F(u))_2| \right) \\
= &~{} \inf_{x\in \R^d} \left( \left|  e^{t\Delta}g_1+\int_{0}^{t} e^{(t-s)\Delta} (L_m (u))_1ds -1 \right|+ \left| e^{t\Delta}g_2+\int_{0}^{t} e^{(t-s)\Delta} (L_m (u))_2ds \right| \right) \\
\geq  &~{} \inf_{x\in \R^d} \left( \left|  e^{t\Delta}g_1 -1 \right|+ \left| e^{t\Delta}g_2 \right| \right) \\
&~{} -\sup_{x\in \R^d} \left( \left|  \int_{0}^{t} e^{(t-s)\Delta} (L_m (u))_1ds  \right| +\left|  \int_{0}^{t} e^{(t-s)\Delta} (L_m (u))_2ds  \right|  \right) .
\end{aligned}
\]
Using \eqref{paso1} and Lemma \ref{lemma:boundL},
\[
\begin{aligned}
P(F(u(t))) \geq &~{} \dfrac{3\varepsilon}{2}-2TM_{1,\beta,\varepsilon,m}.
\end{aligned}
\]
Finally, choosing $T\leq \dfrac{\varepsilon}{4 M_{1,\beta,\varepsilon,m}}$, one has
\[
\inf_{x\in \R^d}P(F(u(t)))\geq \varepsilon,  
\]
for each $t\in [0,T]$. Hence $F(\mathcal B)\subseteq \mathcal B$.
 
 \medskip
 
The fact that $F$ is a contraction is classical. Indeed, let $u,v\in \mathcal B$. Note that
\[
F(u(t))-F(v(t))=\lambda \int_{0}^te^{(t-s)\Delta} (L_{m}(u(s))-L_{m}(v(s))) ds.
\]
Consequently, using Lemma \ref{lemma:boundL},
\[
\|F(u(t))-F(v(t))\|_Y\leq \int_{0}^t \|( L_{m}(u)-L_{m}(v))\|_X ds=T\|L_{m}(u)-L_{m}(v)\|_X.
\]
Keeping in mind Lemma \ref{lemma:boundL}  we have
\[
\|F(u(t))-F(v(t))\|_Y\leq TM_{2,\beta,\varepsilon,m}\|u-v\|_X.
\]
Finally, choosing $T\leq \dfrac{1}{2 M_{2,\beta,\varepsilon,m}}$, the contraction is ensured. %the  proof is concluded.

\medskip

\noindent
\emph{End of proof of Theorem \ref{LDirichlet}}. Let
\[
T:=\min\left\{\dfrac{1}{2M_{2,\beta,\varepsilon,m}},\dfrac{\beta}{2M_{1,\beta,\varepsilon,m}},\dfrac{\varepsilon}{4M_{1,\beta,\varepsilon,m}}\right\}.
\]
Note that $\mathcal B$ is a closed subset of a complete space, in particular $\mathcal B$ is complete. Thus by the {Banach fixed point theorem}, we get that for each $\ t\in[0,T]$, there exists a unique $u\in X$ solving the equation $u(t)=F(u(t))$, namely \eqref{eqn:edp} in the Duhamel sense.

\begin{remark} Theorem \ref{LDirichlet} can be extended with some extra work to any PDE model of the form \eqref{eqn:edp} with a source function having at most a finite number of poles, by choosing the initial condition with values that are sufficiently far from the poles. This approach uses the boundedness of meromorphic functions via non-explicit bounds. We thank Felipe Gonçalves for this remark.
\end{remark}

\subsection{Global existence vs. Blow-up dichotomy} Now we show the blow-up alternative \eqref{dicotomia}. Assume that the maximal time of existence of a given solution $u$ is $0<T<+\infty$. If for every $x\in\R^d$ one has
\[
\liminf_{t\uparrow T^*} |u_1(t,x) -1| +|u_2(t,x)|>0, \quad \hbox{ and } \quad \limsup_{t\uparrow T^*} |u(t,x)| <+\infty,
\] 
then a classical continuity argument starting with initial data $u(t)$ at time $t=T^*-\varepsilon$, $\varepsilon>0$ sufficiently small, allows one to continue the solution further in time with $P(u)(t)>0$ and $u(t)$ bounded, up to time $t>T^*$, contradicting the definition of maximal time of existence $T^*$.

%The Theorem \ref{LDirichlet} can be relaxed, for any $g\in Y$, for the following.
%
%
%
%
%\begin{cor}
%Let $g\in Y$ y $L_{m}(s)$ associated to a non principal character, therefore the problem
%\[
%\begin{aligned}%\tag{\ref{eqn:edp3}} 
%\partial_t u(t,x)&=\Delta u(t,x)+ \lambda L_{m}(u(t,x))\hspace{2cm}x \in \R^d \quad t\ge 0\\
%u(0,x)&=g(x)\hspace{2cm}\forall x\in \R^d,   %\nonumber
%\end{aligned}
%\]
%with $L_{m}(s)=\sum_{n=1}^\infty \dfrac{\chi_{m}(n)}{n^s}$ and $\chi_{m}(n)$ some non principal character with $m$-period, has a unique solution for al $t\in[0,T]$, with $T$ small.
%\end{cor}

\section{Global well-posedness}\label{sec:5}

In this section, we prove Theorems \ref{teo:globalpos} and \ref{teo:globalReal}, and Corollary \ref{Cor:faible}. From Remark \ref{caso_zeta_VDL}, we know the existence of a $\sigma_1>1$ such that $\zeta(\sigma_1)=2$. Indeed, $\sigma_1\sim 1.71$.

\subsection{Proof of Theorem \ref{teo:globalpos}} We start with the following simple result that justifies Remark \ref{caso_general_VDL} in the introduction. Notice that in general, $\sigma_0(L_m)$ need not be larger than 1.

\begin{lemma}\label{lem:Repos}
Let $L_m$ be any Dirichlet $L$-function and $\sigma_0=\sigma_0(L_m)$ be as in Definition \ref{sigma0}. Let $s\in \Com$, $s=s_1+is_2$. Then the following statements  are satisfied:
\begin{enumerate}
\item[$(i)$] If $s_1 > \sigma_0$, then $\re(L_m(s))>0$.
\item[$(ii)$] If now $s_1>1$, then 
\[
\re(L_m(s))> 2-\zeta(s_1).
\]
\item[$(iii)$] If in addition $s_1\geq \sigma_1>1$, then one has $2-\zeta(s_1)\geq 0$.
\end{enumerate}
\end{lemma}
\begin{proof}
Proof of $(i)$. Clearly $\re(L_m(s))\neq 0$ if $s_1>\sigma_0$. By continuity and since $\re(L_m(s))$ tends to 1 as $s_1$ tends to plus infinity, we obtain the desired inequality.

\medskip

Proof of $(ii)$. First of all, note that in the domain $s_1=\re(s)>1$ the representation \eqref{L_m_serie} for $L_m$ is valid and
 \[
 \re(L_m(s))= \displaystyle{\re\left(\sum_{n\geq 1}\dfrac{\chi_m(n)}{n^s}\right)=1+\re\left(\sum_{n\geq 2}\dfrac{\chi_m(n)}{n^s}\right)}. 
 \]  
If $m\neq 1$, from Remark \ref{propiedades_chi} there exists at least one $\chi_m(n)$ that is identically $0$. Since $|n^s| = n^{s_1}$ and $|\chi_m(n)| \leq 1$, one gets 
 \[
 \re(L_m(s)) \geq 1- \sum_{n\geq 2}\dfrac{|\chi_m(n)|}{|n^s|} > \displaystyle{1-\sum_{n\geq 2}\dfrac{1}{n^{s_1}}=2-\zeta(s_1)}. 
 \]  
Now if $m=1$, from Remark \ref{propiedades_chi} we have $\chi_m(n)=1$ for all $n$ and $L_m=\zeta$. We first obtain
 \[
 \begin{aligned}
 \re(L_m(s))= &~{}  \displaystyle{1+\re\left(\sum_{n\geq 2}\dfrac{1}{n^s}\right)=1+\sum_{n\geq 2}\dfrac{\cos(s_2\log(n))}{n^{s_1}}}. 
 \end{aligned}
 \]  
 Additionally, there is no $s_2\in\R$ such that $\cos(s_2\log(2))=\cos(s_2\log(3))=-1$ (since $\frac{2k_1+1}{\log(2)} \neq \frac{2k_2+1}{\log(3)}$ for all $k_1,k_2\in\Z$). Therefore,
 \[
 \begin{aligned}
 \re(L_m(s)) > &~{} 1-\left(\sum_{n\geq 2}\dfrac{1}{n^{s_1}}\right)=2-\zeta(s_1). 
 \end{aligned}
 \]  
This proves $(ii)$. 

\medskip

Proof of $(iii)$. Finally, if now $s_1=\re(s)\geq \sigma_1$, one has $\re(L_m(s))> 2-\zeta(s_1) \geq 0$, as desired. This ends the proof.
\end{proof}

%{\color{blue}
%\begin{remark}
%If $s_1>\sigma_0$ you have that $\re(L_m(s))>0$
%\end{remark}
%}

Now we will continue the proof of Theorem \ref{teo:globalpos}. Thanks to  \eqref{cond_0}, the condition \eqref{uniform_g} is also satisfied,  then the hypotheses in Theorem \ref{LDirichlet} are valid, and also its conclusion. Then there exists a unique solution   $u$ of the problem \eqref{eqn:PDZN} for $T>0$ small enough.

\medskip

Recall $I_1\in\R$ defined in \eqref{I_S}. Now condition \eqref{cond_0} implies $ I_1>\sigma_0$. From Lemma \ref{lem:Repos} (i), we know that $\re(L_m(s))>0$ when $s_1\geq I_1$. 
By continuity, we still can take $\varepsilon<0$ such that $\re(L_m(s))>0$ whenever $s_1\geq I_1+ \varepsilon>\sigma_0$. Consider the (constant) function  $V(t,I_1,\varepsilon)= I_1+ \varepsilon$. $V$ solves the ODE in \eqref{heat2} with $G\equiv 0$ and $V_0=I_1$. Following Lemma \ref{comparisonprinciple}, $F_1(V(t,V_0,\varepsilon),u_2)$ is given by
\[
\begin{aligned}
F_1(V(t,I_1,\varepsilon),u_2)  = &~{}  \re L_m(V(t,I_1,\varepsilon)+ i u_2) \\
=&~{} \re L_m(I_1+ \varepsilon + i u_2) >0 = G(V(t,I_1,\varepsilon)),
\end{aligned}
\]
since $I_1+ \varepsilon>\sigma_0$. Additionally, $u_{0,1} (x)= g_1(x) \geq I_1 >I_1+\varepsilon $. Therefore, Lemma \ref{comparisonprinciple} (ii) implies that 
 \begin{equation}\label{u1geqI1}
 u_1(t,x)\geq I_1,\quad \text{for all } (t,x)\in [0,T]\times \R^d .
 \end{equation}
Notice that $u$ will never touch the pole. We prove now the following improvement:
    \begin{lemma}
    Under \eqref{cond_0} one has
    \begin{equation}\label{eqn:realpart}
       \max\{0, 2-\zeta(I_1)\}< \re(L_m(u(t,x)))<\zeta(I_1),
    \end{equation}
    and
        \begin{equation}\label{eqn:imagpart}
  \left| \ima(L_m(u(t,x))) \right| < \zeta(I_1)-1.
    \end{equation}
    \end{lemma}

\begin{proof}    
Notice that from \eqref{u1geqI1} we know that $u_1(t,x) >I_1>\max\{1,\sigma_0(L_m)\}$,  and  we have the simple representation of $L_m$:
\[ 
\re(L_m(u(t,x)))=\displaystyle{\re\left(\sum_{n\geq 1}\dfrac{\chi_{m}(n)}{n^{u(t,x)}}\right)< \sum_{n\geq 1}\dfrac{1}{n^{u_1(t,x)}}=\zeta(u_1(t,x))}\leq \zeta(I_1).
\]
This proves the second inequality in \eqref{eqn:realpart}. Now we bound by below $\re (L_m(u(t,x)))$. Notice that $\re u(t,x)>1$. Invoking Lemma \ref{lem:Repos} (ii) and \eqref{u1geqI1},% And for the left:
\[ 
\re(L_m(u(t,x))) \geq 2- \zeta (u_1(t,x))
%=\displaystyle{1+\re\left(\sum_{n\geq 2}\dfrac{\chi_{m}(n)}{n^{u}}\right)> 1-\sum_{n\geq 1}\dfrac{1}{n^{u_1}}}
\geq 2-\zeta(I_1).
\]
%Since $I_1>\sigma_1$, we obtain \eqref{eqn:realpart}.
Combining this inequality with Lemma \ref{lem:Repos} (i), we get $\re(L_m(u(t,x))) \geq \max\{0, 2-\zeta(I_1)\}$, as desired.

\medskip

Let us prove \eqref{eqn:imagpart} now.  Similar to the previous bound, but considering that the first term is real valued, %s are 1, we have that 
\[
\left|\ima\left(L_m(u(t,x)) \right)\right|  = \left|\ima\left(\sum_{n\geq 1}\dfrac{\chi_m(n)}{n^{u(t,x)}} \right)\right| = \left|\ima\left(\sum_{n\geq 2}\dfrac{\chi_m(n)}{n^{u(t,x)}} \right)\right|\leq \left| \sum_{n\geq 2}\dfrac{\chi_m(n)}{n^{u(t,x)}} \right|.
\] 
Using that $|\chi_m(n)|\leq 1$ and $|n^s|= n^{s_1}$,
\[
\left|\ima\left(L_m(u(t,x)) \right)\right|\leq \sum_{n \geq 2}\dfrac{|\chi_m(n)|}{|n^{u(t,x)}|} \leq \sum_{n \geq 2}\dfrac{1}{n^{u_1(t,x)}}\leq \zeta(u_1(t,x))-1.
\]
Notice that $\zeta(u_1(t,x))-1>0$ provided $u_1(t,x)>1$. Finally, considering \eqref{u1geqI1}, we conclude the bound \eqref{eqn:imagpart} for the imaginary part of $L_m(u(t,x))$.  The proof is complete.
\end{proof}
   
\noindent   
\emph{End of proof of Theorem \ref{teo:globalpos}.} Recall $I_1$ and $S_1$ as in \eqref{I_S}. Let $\varepsilon_1>0$. Consider the auxiliary problems
\begin{equation}
\left\lbrace\begin{aligned}
\dot V (t)= &~{} \zeta(I_1), \quad  t> 0,\\
V(0) = &~{} S_1+\varepsilon_1,
\end{aligned}\right.
\label{eqn:HeatZetanuevo1}
\end{equation}
and
\begin{equation}
\left\lbrace\begin{aligned}
\dot V (t)= &~{} \max\{0,2-\zeta(I_1)\}, \quad  t> 0,\\
V(0) = &~{} I_1-\varepsilon_1,
\end{aligned}\right.
\label{eqn:HeatZetanuevo2}
\end{equation}
whose globally-defined solutions are trivially found: $V(t)=\zeta(I_1)t+S_1+\varepsilon_1$ and $V(t)= \max\{0,2-\zeta(I_1)\}t+I_1-\varepsilon_1$, respectively.  Using again Lemma \ref{comparisonprinciple} with $\varepsilon=\pm\varepsilon_1$, $G\equiv \zeta(I_1)$ and $G\equiv  \max\{0,2-\zeta(I_1)\}$ respectively, and taking into account \eqref{eqn:realpart}, the solutions of \eqref{eqn:HeatZetanuevo1} and \eqref{eqn:HeatZetanuevo2} are super and subsolutions of $u_1$, respectively. Consequently, Lemma \ref{comparisonprinciple} gives us
\[   
\max\{0,2-\zeta(I_1)\}t+I_1\leq u_1(t,x)\leq \zeta(I_1)t+S_1, \quad  (t,x)\in  [0,T]\times \R^d.
\]
This proves \eqref{des_1}. In a similar fashion, we consider the two problems% the next problems:
\begin{equation}
\left\lbrace\begin{aligned}
\dot V (t)= &~{}(\zeta(I_1)-1), \quad t> 0,\\
V(0) = &~{} S_2.
\end{aligned}\right.
\label{eqn:HeatZetanuevo3}
\end{equation}
and %{\color{red} ojo que la comparacion es EDP EDP aca y no EDP ODE}
\begin{equation}
\left\lbrace\begin{aligned}
\dot V (t)= &~{} (1-\zeta(I_1)), \quad t> 0,\\
V(0) = &~{} I_2.
\end{aligned}\right.
\label{eqn:HeatZetanuevo4}
\end{equation}
Again using the Lemma \ref{comparisonprinciple} and \eqref{eqn:imagpart}, \eqref{eqn:HeatZetanuevo3} and \eqref{eqn:HeatZetanuevo4} are super and subsolutions of $u_2(t,x)$. Therefore, \eqref{des_2} is also satisfied:
\[   
(1-\zeta(I_1))t+I_2\leq u_2(t,x)\leq (\zeta(I_1)-1)t+S_2, \quad  (t,x)\in  [0,T]\times \R^d. 
\]
Finally, these bounds prove that $T$ cannot be a finite maximal time of existence and the solution $u(t,x)$ is global. In particular, if the infimum of the initial real part are larger than  $\max\{ 1,\sigma_0(L_m)\}$, the real part tends to infinity. This concludes the proof of Theorem \ref{teo:globalpos}.

\subsection{Proof of Corollary \ref{Cor:faible}}

We proceed in a similar fashion as in the previous proof. From Theorem \ref{teo:globalpos}, we know that $u$ is global and \eqref{des_1}-\eqref{des_2} are satisfied. Now we improve \eqref{des_2}. In the setting of Lemma \ref{lem:lema2}, consider $M=1$ and the smooth function
    \[
    H_2(u_1,u_2)=- \, u_2.
    \]
Notice that $H_2$ is time independent. Let us introduce the upper half plane in $\mathbb R^2$:% time independent set%    Define the next set:
    \[
    D_2:=\left\{(u_1,u_2)\in\R^2 ~ \Big| ~ H_2(u_1,u_2)<0\right\}=\left\{ (u_1,u_2)\in\R^2 ~ \Big| ~ u_2>0\right\}.
    \]
    Finally, consider $U(t)=(U_1(t),U_2(t))$ solution to \eqref{eqn:edog} with initial datum to be chosen later in $D_2$. We directly have that $\partial D_2= \{ u_2=0\}$ and by reality of the characters of $L_m$, one has $L_m$ real valued on $U_2=0$ and
    \[
    \dfrac{d}{dt}H_2(U_1(t),U_2(t))= -  \dfrac{d}{dt} U_2(t)=-\ima(L_m(U(t)))=0 \quad \text{ on }\partial D_2.
    \]   
    Therefore, Lemma \ref{lem:lema2} implies that $D(t)=D_2$ is invariant under the flow of $U(t)$. Since $\inf_{x\in \R^d}u_2(0,x) =\inf_{x\in \R^d}g_2(x)=I_2>0$, this implies that $\{ u(0,x) ~ : ~ x\in\mathbb R^d \}=\mathcal I[u](0)\subseteq D_2$.  Finally, Lemma \ref{lem:lema1} applies and one gets
    \[
    \mathcal I[u](t) \subseteq D_2,
    \]
    i.e.,  $u_2(t,x)>0$ for all $(t,x)\in  [0,\infty)\times \R^d$.

\subsection{Proof of Theorem \ref{teo:globalReal}} Now we prove global well-posedness in the real-valued case and $L_m=\zeta.$ Recall the conditions on $g$ required in \eqref{1p6_0}.

\medskip

\noindent
\emph{Proof of \eqref{1p6_1}.} The proof is similar to the proof of Theorem \ref{teo:globalpos}. Using the local existence result, Theorem \ref{LDirichlet}, we have that there exists a solution $u=u(t,x)$ for \eqref{eqn:edp} with $\lambda=1$, for all $(t,x)\in [0,T)\times \R^d$, and $T$ a given maximal time of existence.

\medskip

Consider now the real-valued ODE
\[
\begin{aligned}
\dot U(t,I)= &~{} \zeta(U(t,I)), \quad t\ge 0,\\
U(0,I) = &~{} (I,0) =I + 0 i .
\end{aligned}
 \]   
This problem is nothing but equation \eqref{eqn:edog} with $F=\zeta$. Since the initial data is real-valued, $U(t,I)$ is real-valued and with a slight abuse of notation, we understand $U(t)=U_1(t)$. By making $T>0$ smaller if necessary, $U(t,I)$ is well-defined on $[0,T]$. Corollary \ref{comparisonprinciple2} $(i)$ applied to $V=U$ in this case gives
    \[
    U(t,I)\leq u(t,x),\quad \text{for all } (t,x)\in [0,T]\times \R^d.
    \]
However, since $I>1$, $U(t,I)$ is a non decreasing function. Consequently, %, we have that
    \[
    u(t,x)\geq I,\quad \text{for all } (t,x)\in [0,T]\times \R^d .
    \]
    Then 
    \begin{equation}\label{eqn:realpart2}
        1< \zeta(u(t,x))\leq \zeta(I).
    \end{equation}
We can define the next auxiliary problems:
\begin{equation}
\left\lbrace\begin{aligned}
\dot V(t)= &~{}\zeta(I), \quad  t\geq 0,\\
V(0) = &~{} S,
\end{aligned}\right.
\label{eqn:HeatZetanuevo11}
\end{equation}
and
    \begin{equation}
\left\lbrace\begin{aligned}
\dot V(t)= &~{} 1, \quad  t\geq 0,\\
V(0) = &~{} I.
\end{aligned}\right.
\label{eqn:HeatZetanuevo21}
\end{equation}
It is easy to get the solutions to \eqref{eqn:HeatZetanuevo11} and \eqref{eqn:HeatZetanuevo21}. Using Lemma \ref{comparisonprinciple3} and \eqref{eqn:realpart2}, we have that the solutions of \eqref{eqn:HeatZetanuevo11} and \eqref{eqn:HeatZetanuevo21} are super and subsolutions of $u(t,x)$, respectively. Consequently,
\[  
 t+I\leq u(t,x)\leq \zeta(I)t+S, \quad  (t,x)\in  [0,T]\times \R^d.
\]
This proves \eqref{1p6_1}. Finally, from this bound one has that $u(t,x)$ is global and in particular its real part tends to positive infinity. 

\medskip

\noindent
\emph{Proof of \eqref{1p6_2}}. First of all, notice that $I\leq u_0(x)= g(x) \leq S$.  Let $V(t,I,0)= U(t,I)$ and $V(t,S,0)= U(t,S)$ be the ODE solutions as considered in Corollary \ref{comparisonprinciple2} with initial data $I$ and $S$, respectively. Consequently, we have
\begin{equation}\label{eq:edoIS}
U(t,I)\leq u(t,x)\leq U(t,S).
\end{equation}

\begin{remark}
    The global existence of \(U(t,I)\) and \(U(t,S)\) implies the global existence of \(u(t,x)\). Indeed, this result transpires from Corollary \ref{comparisonprinciple2} as well.
\end{remark}

In what follows, we need the following classical auxiliary result.

\begin{lemma}\label{odemonotone}
    Let \(\Omega\) be an open interval on \(\mathbb{R}\). Let \(f \in C^1(\Omega)\) and \(y_0 \in \Omega\). Consider the ODE
    \begin{equation}
        \left\lbrace\begin{aligned}
        \dot y(t) &= f(y(t)), \quad t \geq t_0,\\
        y(t_0) &= y_0.
        \end{aligned}\right.
        \label{eqn:odeauto}
    \end{equation}
    If \(f(y_0) > 0\) (\(f(y_0) < 0\)), then every solution of \eqref{eqn:odeauto} defined on \([t_0, T]\) is strictly monotone increasing (decreasing) inside this interval of existence. Moreover, in both cases, if \(T = \infty\) and \(y(t)\) is bounded, then \(\displaystyle{\lim_{t \to \infty}y(t)}\) is a zero of the function $f$. 
\end{lemma}

\begin{proof}
    Let \(y: [t_0, T] \to \mathbb{R}\) be a solution of \eqref{eqn:odeauto} and assume \(f(y_0) > 0\). It suffices to show that \(f(y(t)) > 0\) for all \(t \in [t_0, T]\). Suppose that \(f(y(t)) > 0\) and \(f(y(t_1)) = 0\), for some \(t_1 \in (t_0, T]\) and \(0 \leq t < t_1\). If \(y_1 = y(t_1)\), then \(f(y_1) = 0\), and hence the constant function \(\hat{y}(t) = y_1\) solves 
    \begin{equation}
        \left\lbrace\begin{aligned}
        \dot y(t) &= f(y(t)), \quad t \geq 0,\\
        y(t_0) &= y_1.
        \end{aligned}\right.
        \label{eqn:odeauto2}
    \end{equation}
    But \(y(t_1 + t - t_0)\) is a solution of \eqref{eqn:odeauto2}, and due to uniqueness enjoyed by both IVPs (since \(f\) is \(C^1\)), one gets \(y(t - t_0 + t_1) = y_1\), a contradiction since \(f(y(t)) > 0\) for \(t\) below but close to \(t_1\). The remaining case is identical.
    %If \(f(y(t_1)) < 0\), for some \(t_1 \in [0, T]\), then due to Intermediate Value Theorem, there would be also a \(t_2 \in [0, T]\), where \(f(y(t_2)) = 0\), again leading to a contradiction.
    
    \medskip
    
    Now we prove the last assertion of the lemma. It is enough to consider the case \(y(t)\) a monotonically increasing function. If \(y(t)\) is bounded, monotonicity implies the existence of the limit. Furthermore, the continuity of \(f\) implies that the limit of \(\dot y(t)\) exists. By using L'Hôpital's rule:
    \[
    \lim_{t \to \infty}y(t) = \lim_{t \to \infty}\dfrac{e^t y(t)}{e^t} = \lim_{t \to \infty}\dfrac{e^t (y(t)+\dot y(t))}{e^t} = \lim_{t \to \infty} \left( y(t) + \dot y(t) \right).
    \]
    Since both limits exist, the limit of \(\dot y(t)\) is 0. Recalling \eqref{eqn:odeauto}, \(\displaystyle{\lim_{t \to \infty}f(y(t)) = 0}\).
\end{proof}

\medskip
The following result considers the classical holomorphic zeta flow, see \cite{BB} for more details.

\begin{lemma}\label{convergenceU}
    Let $n \in \N=\{1,2,3\ldots\}$. Let $U=U(t,y_0)$ be the solution of \eqref{eqn:odeauto} with \(f = \zeta\) and initial datum $U(t_0,y_0)=y_0$. If \(y_0 \in (-4n, -4n+4)\), then \(U(t,y_0) \in (-4n, -4n+4)\) for all $t\geq t_0$
    and \(\displaystyle{\lim_{t \to \infty}U(t,y_0) = -4n+2}\).     
\end{lemma}

\begin{proof}
    If \(y_0 = -4n+2\), it is direct that \(U(t,y_0) = -4n+2\).

    If \(y_0 \in (-4n, -4n+2)\), Remark \ref{signos} implies that \(\zeta(y_0) > 0\) and Lemma \ref{odemonotone} returns that $U$ is monotone increasing. 
    Since $U$ is unique and the fact that the negative even numbers are constant solutions imply that 
    \[
    -4n < U(t,y_0) < -4n+2 \quad \text{for all } t > t_0.
    \]
    Consequently, $U$ is global. Lemma \ref{odemonotone} implies that \(\displaystyle{\lim_{t \to \infty}U(t,y_0) = -4n+2}\).  Finally, if \(y_0 \in (-4n+2, -4n+4)\), the same idea applies, but now \(U(t,y_0) \in (-4n+2, -4n+4) \) is decreasing towards \(-4n+2\).
\end{proof}

\begin{remark}
 This result can be extended to the case \(y_0 \in (0,1)\); in this case, \(\zeta(y_0) < 0\), and using the same argument,  one has that \(\displaystyle{\lim_{t \to \infty}U(t,y_0) = -2}\).
\end{remark}

Now we are ready to prove \eqref{1p6_2} and \eqref{1p6_3}. We will use Lemma \ref{convergenceU} as follows. First, when \(I\) is a negative even number, \(U(t,I) = I\). Let \(k_1\) and \(k_2\) be defined as in \eqref{k1k2}. 
Notice that \(k_1\) is a positive integer. If \(k_1\) is even, \(-2k_1 < I\) is a multiple of 4, and \(I \in (-2k_1, -2k_1+4)\). 
From Lemma \ref{convergenceU}, one gets that \(U(t,I)\) is increasing and converges to the midpoint \(-2k_1+2 = -4n_1+2\). 
On the other hand, if \(k_1\) is odd, it is not a multiple of 4, and \(I \in (-2k_1-2, -2k_1+2)\). 
This implies that \(U(t,I)\) is decreasing and converges to the midpoint \(-2k_1 = -4n_1+2\).  Hence, \(-2k_1 \leq U(t,I)\) and in both cases $U(t,I)$ converges to \(-4n_1+2\). 
Similarly, one has that \(U(t,S) \leq -2k_2\) and $U(t,S)$ converges to \(-4n_2+2\). This and \eqref{eq:edoIS} implies that 
%\[
%-2k_1 \leq U(t,I) \leq U(t,S) \leq -2k_2.
%\]
%Finally, using the equation \eqref{eq:edoIS} one gets 
\[
-2k_1 \leq U(t,I) \leq u(t,x) \leq U(t,S) \leq -2k_2,
\]
proving \eqref{1p6_2}. Finally, taking the limit one obtains 
\[
-4n_1+2 \leq \liminf_{t\to \infty}u(t,x) \leq \limsup_{t\to \infty}u(t,x) \leq -4n_2+2,
\]
proving \eqref{1p6_3}.

\section{Proof of Theorem \ref{ThmZS}}\label{estabilidad}

Now we prove the stability of stable Riemann zeta zeros. Let $z_0\in \mathbb C$ be a stable Riemann zeta zero. 

\medskip

Assume that $u$ is a solution defined in $[0,T)$ and issued from the initial datum $g$ in the disc $D(z_0,\delta)$, as expressed in Theorem \ref{ThmZS}. Consider the setting of Lemma \ref{lem:lema2}, with $M=1$ and $\delta>0$ a small number to be chosen later. Let $H_3$ be the smooth function
\begin{equation}\label{H3}
 H_3(u_1,u_2)=(u_1-\re(z_0))^2+(u_2-\ima(z_0))^2- \delta^2.
\end{equation}
Notice that $H_3$ is again time independent. Notice that the complex disc $D(z_0,\delta)$ centered on $z_0$ and with radius $\delta$, seen as a subset of $\mathbb R^2$, can be written in terms of $H_3$ as follows:
\[
\begin{aligned}
    D_3:&=\left\{(u_1,u_2)\in \R^2 ~ \Big| ~ H_3(u_1,u_2)<0\right\}\\
    &=\left\{(u_1,u_2)\in \R^2 ~ \Big| ~ (u_1-\re(z_0))^2+(u_2-\ima(z_0))^2<\delta^2\right\} = D(z_0,\delta).
\end{aligned}
    \]
    Consider now $U(t)=(U_1(t),U_2(t))$ solution to \eqref{eqn:edog} with initial datum to be chosen later inside $D_3$. We directly have that 
    \[
    \partial D_3= \left\{ \big( \re(z_0)+\delta \cos(\theta),\ima(z_0)+ \delta  \sin(\theta) \big) ~ \Big| ~ \theta\in[0,2\pi)]\right\}.
    \]
Since $\partial D_3$ is far from the pole $s=1$ for any stable zero of the Riemann zeta and for $\delta>0$ sufficiently small, from the analyticity of $\zeta$ we obtain the following approximation of $\zeta$ on the boundary of $D_3$, which is a compact set in $\R^2$:
\begin{equation}\label{approx}
\begin{aligned}
    \zeta \left( z_0+\delta e^{i\theta} \right)&=\zeta(z_0)+\zeta'(z_0) \delta e^{i\theta}+O(\delta^2)\\
    &=\delta \big( \re(\zeta'(z_0))\cos(\theta)-\ima(\zeta'(z_0))\sin(\theta) \big)\\
    & \quad +i \delta \big( \re(\zeta'(z_0))\sin(\theta)+\ima(\zeta'(z_0))\cos(\theta) \big)+O(\delta^2).   
\end{aligned}
\end{equation}
The term $O(\delta^2)$ represents a quantity bounded by $C\delta^2$, $C>0$ fixed and uniformly on $\partial D_3$. Calculating now the time derivative of $H_3(U_1(t),U_2(t))$ on the boundary, and using \eqref{H3} and \eqref{approx}, we get
\[
U_1(t)-\re(z_0) = \delta \cos(\theta(t)), \quad U_2(t)-\ima(z_0) = \delta \sin (\theta(t)),
\]
\[
\re (\zeta(U(t))) =\delta \big( \re(\zeta'(z_0))\cos(\theta(t))-\ima(\zeta'(z_0))\sin(\theta(t)) \big) + O(\delta^2),
\]
\[
\ima (\zeta(U(t))) =\delta \big( \re(\zeta'(z_0))\sin(\theta(t))+\ima(\zeta'(z_0))\cos(\theta(t)) \big)+O(\delta^2),
\]
and
   \[
    \begin{aligned}
        \dfrac{d}{dt}H_3(U_1(t),U_2(t))&=2\Big((U_1(t)-\re(z_0))\dot U_1(t) +(U_2(t)-\ima(z_0))\dot U_2(t) \Big)\\
        &=2\delta \Big( \cos(\theta(t))\re(\zeta(U(t))) +\sin(\theta(t))\ima(\zeta(U(t))) \Big)\\
        &=2\delta^2\re(\zeta'(z_0))+O(\delta^3).
    \end{aligned}    
    \]   
    Taking $\delta$ sufficiently small, and since $\re(\zeta'(z_0))<0$, we have that
    \[
    \dfrac{d}{dt}H_3(U_1(t),U_2(t))<0.
    \]
    Therefore, Lemma \ref{lem:lema2} implies that $D(t)=D_3$ is invariant under the flow of $U(t)$. Recall that by hypothesis $\{ g(x)=u(0,x) ~ : ~ x\in\mathbb R^d \}=\mathcal I[u](0)\subseteq D_3$. Then Lemma \ref{lem:lema1} applies and one gets
    \[
    \mathcal I[u](t) \subseteq D_3,
    \]
    i.e.,  $u(t,x)\in D(z_0, \delta)$ for all $(t,x)\in  [0,T)\times \R^d$.  This proves that $T=+\infty$ and $u(t, x)$ is global in time and bounded. 
    
\medskip    
    
 Now we prove the asymptotic stability. Let us consider a similar setting as before with a slight difference: let $\widehat{H_3}$ be given by
\[
  \widehat{H}_3(t,u_1,u_2):=(u_1-\re(z_0))^2+(u_2-\ima(z_0))^2-\delta^2e^{-t \delta^2}.
\]
In this case, we obtain 
\[
    \begin{aligned}
    \widehat{D_3}(t):&=\left\{(u_1,u_2)\in \R^2 ~ \Big| ~ \widehat{H_3}(t,u_1,u_2)<0\right\}\\
    &=\left\{(u_1,u_2)\in \R^2 ~ \Big| ~ (u_1-\re(z_0))^2+(u_2-\ima(z_0))^2<\delta ^2e^{-t\delta^2}\right\}.
    \end{aligned}
    \]
    Considering the same arguments as before, we show that the derivative of $\widehat{H_3}$ on the boundary is given by
    \[
    \begin{aligned}
        \dfrac{d}{dt}\widehat{H_3}(t,U_1(t),U_2(t))&=2((U_1(t)-\re(z_0))\dot U_1(t) +(U_2(t)-\ima(z_0))\dot U_2(t))-\delta^4e^{-t\delta^2}\\
        &=2\delta e^{-t \delta} \Big( \cos(\theta(t))\re(\zeta(U(t))) +\sin(\theta(t))\ima(\zeta(U(t))) \Big)- \delta^4e^{-t \delta^2}\\
        &=\delta e^{-t \delta^2}(2\delta \re(\zeta'(z_0))+O(\delta^2)).     
    \end{aligned}    
    \]   
Again take $\delta>0$ sufficiently small; since $\re(\zeta'(z_0))<0$, we have that
    \[
    \dfrac{d}{dt}\hat{H_3}(U_1(t),U_2(t))<0.
    \]
    Therefore, Lemma \ref{lem:lema2} implies that $D(t)=\widehat{D_3}(t)$ is invariant under the flow of $U(t)$. Exactly as before we have for all $t\geq 0$ %Recall that $\{ u(0,x) ~ : ~ x\in\mathbb R^d \}=\mathcal I[u](0)\subseteq \widehat{D_3}(0)$, the Lemma \ref{lem:lema1} applies and one gets
    \[
    \mathcal I[u](t) \subseteq\widehat{ D_3}(t).
    \]
Finally, since by classical topological arguments $\displaystyle{\bigcap_{t\geq 0}\widehat{D}_3(t)=\{z_0\}}$, we obtain $\displaystyle{\lim_{t \to \infty}u(t,x)=z_0}$, the desired result. This ends the proof of Theorem \ref{ThmZS}.

\section{Blow up}\label{sec:6}

Now we prove Theorem \ref{ThmC5}. First we note that Theorem \ref{LDirichlet} ensures the existence of a time $0<T_1\leq \infty$ such that $u$ solves \eqref{eqn:PDZR-}. Note additionally that $u$ is real-valued. Assume that $T_1=+\infty$. 

\medskip

Case $I>1$. Consider  $V(t)=S-t$ global solution to the auxiliary problem
\[%begin{equation}
\left\lbrace\begin{aligned}
\dot V(t)= &~{}-1, \quad t\geq 0,\\
V(0) = &~{} S.
\end{aligned}\right.
%\label{eqn:HeatZetamenos}
\]%end{equation}
Since $S\geq I>1$, and $u$ is real-valued, we get $1<u(t,x)$. Hence, $\zeta(u(t,x)) >1$ and $-\zeta(u(t,x))<-1$ as long as $t<T_1$. Then Corollary \ref{comparisonprinciple3} applied with $V(t)$ gives
\[
1< u(t,x)\leq S-t,
\]
for each $t\in [0,\infty)$. However, the previous inequality is only valid if $T_1<S$, %.  If $T_1=\infty$, we have that $u$ is global and by continuity   $u(t,x)>1$, for all $(t,x)\in [0,\infty)\times \R^d$,
which is a contradiction. 

\medskip

%\begin{remark}
%    Theorem \ref{ThmC5} can be extended to the case of initial conditions satisfying $-2 <I <S <1$. %And we like to obtain a time more exactly. Another question is if we have a global blow up (on all space) or only in a point.
%    \end{remark}
    
Case $-2 <I <S <1$.   Since $-2 < I \leq u(0,x) \leq S < 1$, using that $u(t,x) = -2$ is a solution of \eqref{eqn:PDZR-}, we get $1 > u(t,x) > -2$. Consider $V_1(t) = I$ global solution to the auxiliary problem
    \[
    \left\lbrace\begin{aligned}
    \dot V(t) = &~{} 0, \quad t \geq 0,\\
    V(0) = &~{} I.
    \end{aligned}\right.
    \]
    Hence, from Remark \ref{signos} $u(t,x)$ satisfies $-\zeta(u(t,x)) \geq 0$. Then Corollary \ref{comparisonprinciple3} applied with $V_1(t)$ gives
    \[
    I < u(t,x) < 1,
    \]
    for each $t \geq 0$. Now we consider $V_2(t) = I - \zeta(I)t$ solving
    \[
    \left\lbrace\begin{aligned}
    \dot V(t) = &~{} -\zeta(I), \quad t \geq 0,\\
    V(0) = &~{} I.
    \end{aligned}\right.
    \]
    Since $I \leq u(t,x) < 1$, the monotonicity of $\zeta$ in this interval implies that $0 < -\zeta(I) \leq -\zeta(u(t,x))$. Then Corollary \ref{comparisonprinciple3} applied this time with $V_2(t)$ gives
    \[
    I - \zeta(I)t \leq u(t,x) < 1.
    \]
    However, the previous inequality is only valid if $T_2 < \dfrac{I}{\zeta(I)}$, which is a contradiction.

\medskip

Case $S<-2$. The proof is similarly to the Theorem \ref{teo:globalpos}. Indeed, one has
\begin{equation}\label{1p6_2_bis}
-2k_1\leq u(t,x)\leq -2k_2,
\end{equation}
where $k_1$ and $k_2$ are such that
%\[
%k_1:=\inf \left\{k\in \N ~ \big| ~ -2k\leq I\right\} \quad and\quad k_2:=\sup\left\{k\in \N ~ \big| ~ -2k\geq S\right\}. 
%\] 
\[%begin{equation}\label{k1k2-}
\begin{aligned}
-2 k_1:= &~{} \max \left\{-2k \leq I ~ \big|  ~ k\in \N \right\}, \\
 -2 k_2:=&~{}  \min  \left\{-2k \geq S  ~ \big| ~ k \in \N \right\}. 
\end{aligned}
\]%end{equation}
Additionally, if $n_1$ and $n_2$ are the unique positive integers such that $I\in (-4n_1-2,-4n_1+2)$, and $S\in (-4n_2-2,-4n_2+2)$ respectively, then
\begin{equation}\label{1p6_3_bis}
-4n_1\leq \liminf_{t\to \infty }u(t,x)\leq \limsup_{t\to \infty }u(t,x)\leq -4n_2.
\end{equation}
The proof of \eqref{1p6_2_bis} and \eqref{1p6_3_bis} follow the second part of the proof of Theorem \ref{teo:globalReal}. In this case, we use the following result:
\begin{lemma}%\label{convergence-U}
    Let $n \in \N=\{1,2,3\ldots\}$. Let $U=U(t,y_0)$ be the solution of \eqref{eqn:odeauto} with \(f = -\zeta\) and initial datum $U(t_0,y_0)=y_0$. If \(y_0 \in (-4n-2, -4n+2)\), then \(U(t,y_0) \in (-4n-2, -4n+2)\) for all $t\geq t_0$
    and \(\displaystyle{\lim_{t \to \infty}U(t,y_0) = -4n}\).     
\end{lemma}
\begin{proof}
    If \(y_0 = -4n\), it is direct that \(U(t,y_0) = -4n\).  If \(y_0 \in (-4n-2, -4n)\), Remark \ref{signos} implies that \(-\zeta(y_0) > 0\) and Lemma \ref{odemonotone} returns that $U$ is monotone increasing. 
    Since $U$ is unique and the fact that the negative even numbers are constant solutions imply that 
    \[
    -4n-2 < U(t,y_0) < -4n \quad \text{for all } t > t_0.
    \]
    Consequently, $U$ is global. Lemma \ref{odemonotone} implies that \(\displaystyle{\lim_{t \to \infty}U(t,y_0) = -4n}\).  Finally, if \(y_0 \in (-4n, -4n+2)\), the same idea applies, but now \(U(t,y_0) \in (-4n, -4n+2) \) is decreasing towards \(-4n\).
\end{proof}

\appendix
%{\color{blue} Estoy ordenando estos para que estes mas presentables y recortar los pasos innecesarios  de aqui al 30 dic deberian quedar listos el A y B.} 
\section{Proof of Lemma \ref{lemma:boundh}}\label{app:A}

Now we prove explicit bounds on the functions involved in Lemma \ref{lemma:boundh}. %The proof of Lemma \ref{lemma:boundh} will be divided in two steps.

\medskip

\noindent
{\bf Step 1. Bound on $h(s,\alpha)$.} Recall that from \eqref{h} one has
\[
h(s,\alpha)= 2 \int_0^{\infty} \frac{\sin \left(s \arctan \frac{t}{\alpha}\right)}{\left(\alpha^2+t^2\right)^{\frac{s}{2}}\left(e^{2 \pi t}-1\right)}\, d t.
\]
We prove 
\begin{lemma}%[\ref{lemma:boundh}]
For $s \in [-\beta,\beta]^2$, and $\alpha\in (0,1)$,
\[
|h(s,\alpha)|
%=\left|2 \int_0^{\infty} \frac{\sin \left(s \arctan \frac{t}{\alpha}\right)}{\left(\alpha^2+t^2\right)^{\frac{s}{2}}\left(e^{2 \pi t}-1\right)}\, d t\right|
 \leq H_{1,\alpha, \beta}.
\]
Moreover, the constant $H_{1,\alpha,\beta}$ is decreasing in $\alpha$.
\end{lemma}

\begin{proof}
For $t\in[0,1]$, let $\omega=\omega(t)=\arctan(t/\alpha)\in [0,\pi/2)$. First of all, $\alpha^2\leq \alpha^2+t^2\leq \alpha^2+1$. Consequently, if $s=x+iy$, %(quizás se puede modificar con algo así)
\[
\begin{aligned}
\left| (\alpha^2+t^2)^{s/2} \right|= &~{} \left| e^{\frac{x}{2}\ln(\alpha^2+t^2)} \right| =(\alpha^2+t^2)^{\frac{x}{2}}\\
\geq &~{} \min\left\{ (\alpha^2+1)^{-\frac{\beta}{2}},\alpha^\beta \right\}\geq \left(\dfrac{\alpha^2}{\alpha^2+1}\right)^{\frac{\beta}{2}}.
\end{aligned}
\]
Also,
\[
\begin{aligned}
  |\sin(s\omega)| =&~{} |\sin((x+iy)\omega)|   \\
   \leq &~{} |\sin(x\omega)|+|\sinh(y\omega)| \\
   \leq &~{}  |x|\omega +|\sinh(y\omega)| 
   %+\frac{1}{2} (e^{|y|\omega}-e^{-|y|\omega}) \\
  % &\leq& \beta\omega+\frac{1}{2} (e^{\beta\omega}-e^{-\beta\omega})\\
   \leq  \beta\omega+\sinh(\beta\omega) =: l(\omega).
\end{aligned}
\]
We get
\begin{eqnarray*}
  |h(s, \alpha)| &\leq & 2 \int_0^{\infty} \frac{|\sin \left(s \arctan \frac{t}{\alpha}\right)|}{|\left(\alpha^2+t^2\right)^{\frac{s}{2}}|\left(e^{2 \pi t}-1\right)}\, dt  \\
   &\leq &  2 \left( \int_0^{1} \frac{(1+\alpha^{-2})^{\frac{\beta}{2}} l(\omega) }{   e^{2 \pi t}-1 }\, dt + \int_1^{\infty} \frac{(\alpha^2+t^2)^{\frac{\beta}{2}} l(\omega) }{   e^{2 \pi t}-1 }\, dt \right)=:2(I_1+I_2).
\end{eqnarray*}
For the integral $I_1$ one has the following estimate:
\begin{eqnarray*}
I_1&=& \int_0^{1} \frac{(1+\alpha^{-2})^{\frac{\beta}{2}} l(\omega) }{   e^{2 \pi t}-1 }\, dt\\
&\leq& \int_0^{1} \frac{(1+\alpha^{-2})^{\frac{\beta}{2}} \left(\beta\omega+\sinh(\beta\omega) \right) }{2\pi t}\, dt\\
&\leq& (1+\alpha^{-2})^{\frac{\beta}{2}}\int_0^{1} \frac{ \big(\beta\frac{t}{\alpha}+t\sinh(\frac{\beta}{\alpha}) \big) }{2\pi t}\, dt.
\end{eqnarray*}
We finally obtain
\[ 
I_1\leq  a_{\alpha,\beta}:= \frac1{2\pi }\left(1+\frac1{\alpha^{2}} \right)^{\frac{\beta}{2}}  \left(\frac{\beta}{\alpha}+\sinh \left(\frac{\beta}{\alpha} \right) \right).
\]
Observe that $a_{\alpha,\beta}$ is decreasing in $\alpha$.

For the integral $I_2$ we have the following estimate
\begin{eqnarray*}
  I_2 &=& \int_1^{\infty} \frac{(\alpha^2+t^2)^{\frac{\beta}{2}} l(\omega) }{   e^{2 \pi t}-1 }\, dt \\
   &\leq&   \int_1^{\infty} e^{ -\pi t} (\alpha^2+t^2)^{\frac{\beta}{2}} \left( \beta\frac{\pi}{2} +\sinh\left(\beta\frac{\pi}2 \right) \right)    \, dt \\
   &=& \left(\frac{\beta\pi}{2} +\sinh\left(\beta\frac{\pi}2 \right)\right)\int_0^{\infty}  {(1+t^2)^{\frac{\beta}{2}} }e^{ -\pi t}\, dt \\
   &\leq & \left(\frac{\beta\pi}{2} +\sinh\left(\beta\frac{\pi}2 \right) \right)\int_0^{\infty} (2^{\beta/2}+2^{\beta/2}t^{\beta}+1+t^2)e^{ -\pi t}  \, dt \\
   &=& \left( \frac{\beta\pi}{2}+\sinh\left(\beta\frac{\pi}2 \right)\right)
 \left(\frac{2^{\beta/2}+1}{\pi} + 2^{\beta/2}\frac{\Gamma(\beta+1)}{\pi^{\beta+1}}+\frac{2}{\pi^{3}}\right)=:b_{\beta}.
   \end{eqnarray*}
Therefore
\[
  |h(s,\alpha)|
  %=\left|2 \int_0^{\infty} \frac{\sin \left(s \arctan \frac{t}{\alpha}\right)}{\left(\alpha^2+t^2\right)^{\frac{s}{2}}\left(e^{2 \pi t}-1\right)}\, d t\right| 
  \leq H_{1,\alpha, \beta}:=2(a_{\alpha,\beta}+b_{\beta}). 
  \]
The decreasing character of $H_{1,\alpha, \beta}$ follows directly from the corresponding behavior of $a_{\alpha,\beta}$.
\end{proof}

\noindent
{\bf Step 2. Bounds on $h'(s,\alpha)$.}

\begin{lemma}%\label{lemma:boundhp}
Let $s=x+iy\in [-\beta,\beta]^2$, we have that:
  \[
  |h'(s,\alpha)|\leq H_{2,\alpha,\beta}.
  \]
\end{lemma}

\begin{proof}
Note that:
    \begin{align*}
h(s,\alpha)=&~{}2 \int_0^{\infty} \frac{\sin \left(s \arctan \frac{t}{\alpha}\right)}{\left(\alpha^2+t^2\right)^{\frac{s}{2}}\left(e^{2 \pi t}-1\right)}\, dt .
\end{align*}
Therefore, we have that:
    \begin{align*}
h'(s,\alpha)=&~{} 2\int_0^\infty\dfrac{\arctan\frac{t}{\alpha}\cos(s\arctan\frac{t}{\alpha}))} {(\alpha^2+t^2)^{\frac{s}{2}}(e^{2\pi t}-1)} dt.\\
&-2\int_0^\infty\dfrac{s\sin(s\arctan\frac{t}{\alpha})} {2(\alpha^2+t^2)^{\frac{s}{2}-1}(e^{2\pi t}-1)} dt.\\
= :&~{} I_1+I_2.
\end{align*}
Using that $|\cos(sw)|\leq \cos(w|x|)+\sinh(w|y|)\leq 1+\sinh\left(\beta\frac{\pi}{2}\right)$, 
\begin{align*}    
|I_1|&\leq  \left( 1+\sinh\left(\beta\frac{\pi}{2}\right)\right)\left(\dfrac{2}{\alpha}\right)\int_0^{\infty}\dfrac{\sec(w)^{\beta+2}}{(e^{2\pi \alpha \tan(w)}-1)} dw.\\
&\leq \left( 1+\sinh\left(\beta\frac{\pi}{2}\right)\right)\left(\dfrac{2}{\alpha^{\beta}}\right)I_{\alpha,\beta}.\\
&\leq \left( 1+\sinh\left(\beta\frac{\pi}{2}\right)\right)\left(\dfrac{2}{\alpha^{\beta}}\right)\dfrac{\pi}{2}\left(\dfrac{\sqrt{(\beta+2)^2+4\pi^2\alpha^2}}{2\pi\alpha }\right)^{\beta+2}\left(1+\left(\dfrac{\pi}{2}\right)\dfrac{1}{ e^{\beta+2}-1}\right)= :I_{1,\alpha,\beta}.
\end{align*}
Let us notice that 
\[
|\ln(\sec(w)\alpha)|\leq \ln(\sec(w))+\ln(\alpha^{-1})\leq \sec(w)+\alpha^{-1}.
\]
Consequently,
\begin{align*}
|I_2|&\leq 2\int_{0}^{\frac{\pi}{2}}\dfrac{|\sin(sw)||\alpha^{-s}||\sec(w)^{2-s}|} {(e^{2\pi \alpha \tan(w)}-1)} dw+2\int_{0}^{\frac{\pi}{2}}\dfrac{|\sin(sw)||\alpha^{1-s}||\sec(w)^{3-s}|} {(e^{2\pi \alpha \tan(w)}-1)} dw\\
&\leq \dfrac{2}{\alpha^\beta}\int_{0}^{\frac{\pi}{2}}\dfrac{|\sin(sw)|\sec(w)^{\beta+2}} {(e^{2\pi \alpha \tan(w)}-1)} dw+\dfrac{2}{\alpha^{\beta}}\int_{0}^{\frac{\pi}{2}}\dfrac{|\sin(sw)|\sec(w)^{\beta+3}} {(e^{2\pi \alpha \tan(w)}-1)} dw\\
&\leq \dfrac{4}{\alpha^{\beta}}\int_{0}^{\frac{\pi}{2}}\dfrac{|\sin(sw)|\sec(w)^{\beta+3}} {(e^{2\pi \alpha \tan(w)}-1)} dw.
\end{align*}
Using the same steps as before,% from the proof of Proposition :
\begin{align*}
|I_2|&\leq \dfrac{4}{\alpha^{\beta}}\left(\beta+\dfrac{2}{\pi}\sinh\left(\dfrac{\beta\pi}{2}\right)\right)\int_{0}^{\frac{\pi}{2}}\dfrac{w\sec(w)^{\beta+3}} {(e^{2\pi \alpha \tan(w)}-1)} dw\\
&\leq \dfrac{4}{\alpha^{\beta}}\left(\beta+\dfrac{2}{\pi}\sinh\left(\dfrac{\beta\pi}{2}\right)\right)I_{\alpha,\beta+1}\\
&\leq \dfrac{4}{\alpha^{\beta}}\left(\beta+\dfrac{2}{\pi}\sinh\left(\dfrac{\beta\pi}{2}\right)\right)\dfrac{\pi}{2}\left(\dfrac{\sqrt{(\beta+3)^2+4\pi^2\alpha^2}}{2\pi\alpha }\right)^{\beta+3}\left(1+\left(\dfrac{\pi}{2}\right)\dfrac{1}{ e^{\beta+3}-1}\right)\\
& = : I_{2,\alpha,\beta}.
\end{align*}
Finally with this bounds, we have that:
\[
|h'(s,\alpha)|\leq |I_1|+|I_2|\leq I_{1,\alpha,\beta}+I_{2,\alpha,\beta}=:D_{1,\alpha,\beta}.
\]
The proof is complete.
\end{proof}

\section{Proof of Lemma \ref{lemma:boundd}}\label{app:B}

Recall from \eqref{d} that $d(s,\alpha)=\frac{\alpha^{1-s}-1}{s-1}+\frac{1}{2\alpha^s}$. Denote $f(u)=\frac{e^u-1}{u}$, and define $f$ at $u=0$ by continuity. We shall prove the following:

\begin{lemma}%\label{lemma:boundd}
One has
\[
\|d|_{[-\beta,\beta]^2}(s,\alpha)\|_Y\leq D_{1,\alpha,\beta},
\]
and
\[
\|d'|_{[-\beta,\beta]^2}(s,\alpha)\|_Y\leq D_{2,\alpha,\beta}.
\]
\end{lemma}

\begin{proof}
Recall that
\[
d(s,\alpha)=(-\ln(\alpha))f(-\ln(\alpha)(s-1))+ \dfrac{1}{2\alpha^s}.
\]
Therefore
\[
d'(s,\alpha)=(-\ln(\alpha))^2f'(-\ln(\alpha)(s-1))+\dfrac{-\ln(\alpha)}{2\alpha^s}.
\]
Applying the norm and the restriction to $[-\beta,\beta]^2$, we have that:
\begin{align*}
\|d'|_{[-\beta,\beta]^2}(s,\alpha)\|_Y&\leq \left(\ln(\alpha^{-1})\right)^2 \left\| f'|_{[-\beta,\beta]^2}(-\ln(\alpha)(s-1))\right\|_{Y}+\dfrac{\ln(\alpha^{-1})}{2\alpha^\beta}.\\
&\leq \left(\ln(\alpha^{-1})\right)^2\left\|f'|_{\ln(\alpha^{-1})([-\beta,\beta]^2+1)}(s)\right\|_{Y}+\dfrac{\ln(\alpha^{-1})}{2\alpha^\beta}.
\end{align*}
Finally, using Lemma \ref{prop:boundf}, we have that:
\[
\|d'|_{[-\beta,\beta]^2}(s,\alpha)\|_Y\leq \left(\ln(\alpha^{-1})\right)^2E_{\ln(\alpha^{-1})(\beta+1)}+\dfrac{\ln(\alpha^{-1})}{2\alpha^\beta}= : D_{2,\alpha,\beta}.
\]
The proof is complete.
\end{proof}

\begin{lemma}\label{prop:boundf}
Recall the $Y-$norm defined in Definition \ref{def:YX}. For any $r>0$, one has the bound
    \[
    \|f'|_{[-r,r]^2}\|_Y\leq E_{r},
    \]
   with $E_r:=\dfrac{e^{r}(2r^2+6r+4)}{r^2} $.
\end{lemma}

\begin{proof}
 First, we have that
    \[
    f'(s)=\dfrac{e^{s}(s-1)+1}{s^2}.
    \]
It is known that for any complex-valued function one has
    \[
    \max\{|f_1'(s)|,|f_2'(s)|\}^2\leq f'(s)\overline{f'(s)}=f_1'(s)^2+f_2'(s)^2.
    \]
Also,  for any $A \subseteq \Com$ compact,
\[
\|f'|_{A}\|_Y^2\leq \max_{s\in A}\{f'(s)\overline{f'(s)}\}.
\]
Therefore,
\[
\|f'|_{\partial [-r,r]^2}\|_Y^2\leq \max_{s\in \partial [-r,r]^2}\{f'(s)\overline{f'(s)}\}.
\]
For this $f(s)$ and $s\in \partial [-r,r]^2$, we have that
\[
\begin{aligned}
    f'(s)\overline{f'(s)}=&~{}\dfrac{(e^{s}(s-1)-1)}{s^2}\dfrac{(e^{\overline{s}}(\overline{s}-1)-1)}{\overline{s}^2} \\
    %=&~{}\dfrac{e^{s}(s-1)e^{\overline{s}}(\overline{s}-1)-e^{s}(s-1)-e^{\overline{s}}(\overline{s}-1)+1}{(x^2+y^2)^2}\\
    %=&~{}\dfrac{e^{2x}(x^2+y^2-2x+1)-e^{s}(s-1)-e^{\overline{s}}(\overline{s}-1)+1}{(x^2+y^2)^2}\\
    =&~{}\dfrac{e^{2x}(x^2+y^2-2x+1)-e^{x}(e^{iy}(x-1+iy)+e^{-iy}(x-1-iy))+1}{(x^2+y^2)^2}\\
    =&~{}\dfrac{e^{2x}(x^2+y^2-2x+1)-e^{x}(2\cos(y)(x-1)-2\sin(y)y)+1}{(x^2+y^2)^2}\\
    \leq &~{} \dfrac{e^{2r}(2r^2+2r+1)+e^{r}(2(r+1)+2r)+1}{r^4}\\
    \leq &~{} \dfrac{e^{2r}(2r^2+2r+1)+e^{2r}(4r+2)+e^{2r}}{r^4}\\
    \leq &~{} \dfrac{e^{2r}(2r^2+6r+4)}{r^4}.
\end{aligned}
\]
This implies that
\[
\left\| f'|_{\partial [-r,r]^2}\right\|_Y\leq \dfrac{e^{r}(2r^2+6r+4)}{r^2} =E_r.
\]
Finally, using that the function $f'$ is holomorphic, $f_1'$ and $f_2'$ are harmonic, and
\[
\|f'|_{[-r,r]^2}\|_Y= \|f'|_{\partial [-r,r]^2}\|_Y.
\]
We conclude that
\[
\|f'|_{[-r,r]^2}\|_Y= \|f'|_{\partial [-r,r]^2}\|_Y\leq E_{r}.
\]
This ends the proof of the lemma.
\end{proof}

\bibliographystyle{unsrtnat}

%\bibliography{library}
%\bibliography{biblio}

\end{document}